\documentclass[12pt]{amsart}
\usepackage{amscd,amssymb,amsthm,amsmath,amssymb,mathrsfs,enumerate,apptools}
\usepackage[matrix,arrow,curve]{xy}
\usepackage[margin=1.8cm]{geometry}

\AtAppendix{\counterwithin{lemma}{section}}
\AtAppendix{\counterwithin{corollary}{section}}
\newtheorem{theorem}{Theorem}
\newtheorem{proposition}{Proposition}
\newtheorem{lemma}{Lemma}
\newtheorem{corollary}{Corollary}

\theoremstyle{definition}

\theoremstyle{remark}
\newtheorem{remark}{Remark}

\author{Ivan Cheltsov}

\address{\emph{Ivan Cheltsov}
		\newline
		\textnormal{University of Edinburgh,  Edinburgh, Scotland}
		\newline
		\textnormal{\texttt{I.Cheltsov@ed.ac.uk}}}

\title{K-stability of Fano 3-folds of Picard rank 3 and degree 22}

\thanks{Throughout this paper, all varieties are assumed to be projective and defined over~$\mathbb{C}$.}

\dedicatory{To the memory of Sasha Ananin}

\begin{document}

\begin{abstract}
We prove K-stability of smooth Fano 3-folds of Picard rank 3 and degree 22 that satisfy very explicit generality condition.
\end{abstract}

\maketitle

\tableofcontents

\section{Introduction}
\label{section:intro}

Let $L$ is a line in $\mathbb{P}^3$, let $C_4$ be a smooth quartic elliptic curve in $\mathbb{P}^3$ such that $L\cap C_4=\varnothing$,~and
let~$\pi\colon X\to\mathbb{P}^3$ be the blow up of these two  curves.
Then $X$ is a smooth Fano 3-fold of degree~22.
Moreover, all smooth Fano 3-folds of Picard rank 3 and degree 22 can be obtained in this way.

Choosing appropriate coordinates $x_0$, $x_1$, $x_2$, $x_3$ on $\mathbb{P}^3$,
we may assume that
$$
C_4=\big\{x_0^2+x_1^2+\lambda(x_2^2+x_3^2)=0,\lambda(x_0^2-x_1^2)+x_2^2-x_3^2=0\big\}\subset\mathbb{P}^3
$$
for a complex number $\lambda$ such that $\lambda\not\in\{0,\pm 1,\pm i\}$. Further, we may assume that
$$
L=\big\{a_0x_0+a_1x_1+a_2x_2=0,b_1x_1+b_2x_2+b_3x_3=0\big\}\subset\mathbb{P}^3
$$
for some $[a_0:a_1:a_2]$ and $[b_1:b_2:b_3]$ in $\mathbb{P}^2$. The following result is proved in \cite{Book}.

\begin{lemma}[{\cite[Lemma 5.73]{Book}}]
\label{lemma}
If $L=\{x_0-x_2=0,x_1-x_3=0\}$, then $X$ is K-stable.
\end{lemma}

This implies that the Fano 3-fold $X$ is K-stable if $L$ and $C_4$ are chosen to be sufficiently general, because K-stability is an open property \cite[Theorem 4.5]{BlumLiuXu} (see \cite{Xu} for basics facts about K-stability).
Actually, we expect that $X$ is always K-stable. However, we are unable to prove this~at~the~moment.
In this paper, we prove that $X$ is K-stable if it satisfies the following condition:
\begin{center}
for every plane $\Pi\subset\mathbb{P}^3$ passing through $L$,\\
the intersection $\Pi\cap C_4$ contains at most one multiple point, \\
and the multiplicity of this point is at most three.
\end{center}
One can check that this generality condition holds in case when $L$ is the line $\{x_0-x_2=0,x_1-x_3=0\}$.
Hence, our main theorem is a generalization of Lemma~\ref{lemma}, but their proofs are very different.

To state our main theorem in a more natural~way, note that we have commutative diagram
$$
\xymatrix{
&&\mathbb{P}^1\times\mathbb{P}^1\ar@/^1pc/@{->}[drr]^{\mathrm{pr}_2}\ar@/_1pc/@{->}[dll]_{\mathrm{pr}_1}&&\\%
\mathbb{P}^1&&X\ar@{->}[d]_{\pi}\ar@{->}[u]_{\eta}\ar@{->}[rr]^{\phi}\ar@{->}[ll]_{\sigma}&&\mathbb{P}^1\\%
&&\mathbb{P}^3\ar@{-->}[urr]_{\varphi}\ar@{-->}[ull]^{\varsigma}&&}
$$
where $\varsigma$ is given by $[x_0:x_1:x_2:x_3]\mapsto[a_0x_0+a_1x_1+a_2x_2:b_1x_1+b_2x_2+b_3x_3]$,
the map $\varphi$ is given~by
$$
[x_0:x_1:x_2:x_3]\mapsto\big[x_0^2+x_1^2+\lambda(x_2^2+x_3^2):\lambda(x_0^2-x_1^2)+x_2^2-x_3^2\big],
$$
the map $\sigma$ is a fibration into quintic del Pezzo surfaces,
$\phi$ is a fibration into sextic del Pezzo surfaces,
the map $\eta$ is a conic bundle, $\mathrm{pr}_1$ and $\mathrm{pr}_2$ are projections to the first and the second factors, respectively.
Now, we can state our main theorem as follows:

\begin{theorem}
\label{theorem:main}
Suppose that every singular fiber of the fibration $\sigma\colon X\to\mathbb{P}^1$ has one singular~point,
and this point is either a singular point of type $\mathbb{A}_1$ or a singular point of type $\mathbb{A}_2$.
Then $X$ is K-stable.
\end{theorem}

Using basic geometric facts about quintic del Pezzo surfaces with at most Du Val singularities \cite{CorayTsfasman},
one can show that the generality condition in Theorem~\ref{theorem:main} is equivalent to the following condition:
\begin{center}
every fiber of the conic bundle $\eta\colon X\to\mathbb{P}^1\times\mathbb{P}^1$ is reduced.
\end{center}
But this condition is equivalent to the smoothness of the discriminant curve of the conic bundle~$\eta$.

\begin{corollary}
\label{corollary:main}
If the discriminant curve of the conic bundle $\eta$ is smooth, then $X$ is K-stable.
\end{corollary}

One can show that the discriminant curve of the conic bundle $\eta$ is given in $\mathbb{P}^1\times\mathbb{P}^1$ by
\begin{multline*}
\quad\quad \quad \lambda(b_1^2-\lambda b_2^2+\lambda b_3^2)y_0^2z_0^3-(\lambda^3b_2^2+\lambda^3 b_3^2+b_1^2)y_0^2z_0^2z_1-(\lambda^3 b_1^2-b_2^2+b_3^2)y_0^2z_0z_1^2-\\
-(\lambda^3b_1^2-b_2^2+b_3^2)y_0^2z_1^3-2\lambda(a_1b_1-\lambda a_2b_2)y_0y_1z_0^3+2(\lambda^3a_2b_2+a_1b_1)y_0y_1z_0^2z_1+\quad\\
\quad+2(\lambda^3a_1b_1-a_2b_2)y_0y_1z_0z_1^2-2\lambda(\lambda a_1b_1+a_2b_2)y_0y_1z_1^3-\lambda(\lambda a_2^2+a_0^2-a_1^2)y_1^2z_0^3-\\
-(\lambda^3 a_2^2+a_0^2+a_1^2)y_1^2z_0^2z_1+(\lambda^3 a_0^2-\lambda^3 a_1^2+a_2^2)y_1^2z_0z_1^2+\lambda(\lambda a_0^2+\lambda a_1^2+a_2^2)y_1^2z_1^3=0,\quad
\end{multline*}
where $([y_0:y_1],[z_0:z_1])$ are coordinates on $\mathbb{P}^1\times\mathbb{P}^1$.

\medskip
\noindent
\textbf{Acknowledgements.}
This work has been supported by EPSRC grant \textnumero EP/V054597/1.

\section{The proof}
\label{section:proof}

Let us use all assumptions and notations introduced in Section~\ref{section:intro}.
To prove Theorem~\ref{theorem:main}, we suppose that each singular fiber of the fibration $\sigma$ has one singular~point,
and this point is either a singular point of type $\mathbb{A}_1$ or a singular point of type $\mathbb{A}_2$.
Set $H=\pi^*(\mathcal{O}_{\mathbb{P}^3}(1))$. Let $E$ and $R$ be the exceptional surfaces of the blow up $\pi$ such that $\pi(E)=C_4$ and $\pi(R)=L$.
Then
\begin{itemize}
\item the quintic del Pezzo fibration $\sigma$ is given by the pencil $|H-R|$,
\item the sextic del Pezzo fibration $\varphi$ is given by the pencil $|2H-E|$,
\item the conic bundle $\eta$ is given by $|3H-E-R|$.
\end{itemize}
Note that $\mathrm{Eff}(X)=\langle E, R, H-R, 2H-E\rangle$,
and the cone $\overline{\mathrm{NE}(X)}$ is generated by the classes of curves contracted by the blow up $\pi\colon X\to\mathbb{P}^3$ and the conic bundle $\eta\colon X\to\mathbb{P}^1\times\mathbb{P}^1$.

By the Fujita--Li valuation criterion \cite{Fujita,Li}, the Fano 3-fold $X$ is K-stable if and only if
$$
\beta(\mathbf{F})=A_X(\mathbf{F})-S_X(\mathbf{F})>0
$$
for every prime divisor $\mathbf{F}$ over $X$, where $A_X(\mathbf{F})$ is the~log discrepancy of the~divisor $\mathbf{F}$, and
$$
S_X\big(\mathbf{F}\big)=\frac{1}{(-K_X)^3}\int\limits_0^{\infty}\mathrm{vol}\big(-K_X-u\mathbf{F}\big)du.
$$
To show this, we fix a prime divisor $\mathbf{F}$ over~$X$.
Then we set $Z=C_{X}(\mathbf{F})$.
If $Z$ is an irreducible~surface, then it follows from \cite{Fujita2016} that $\beta(\mathbf{F})>0$, see also \cite[Theorem 3.17]{Book}.
Therefore, we may assume that
\begin{itemize}
\item either  $Z$ is an irreducible curve in $X$,
\item or $Z$ is a point in $X$.
\end{itemize}
In both cases, we fix a point $P\in Z$. Let $S$ be one of the following two surface:
\begin{enumerate}
\item the surface in the pencil $|H-R|$ that contains $P$,
\item the surface in the pencil $|2H-E|$ that contains $P$.
\end{enumerate}
Then $S$ is a del Pezzo surface with at most Du Val singularities. Set
$$
\tau=\mathrm{sup}\Big\{u\in\mathbb{R}_{\geqslant 0}\ \big\vert\ \text{the divisor  $-K_X-uS$ is pseudo-effective}\Big\}.
$$
For~$u\in[0,\tau]$, let $P(u)$ be the~positive part of the~Zariski decomposition of the~divisor $-K_X-uS$,
and let $N(u)$ be its negative part. If $S\in|H-R|$, then $\tau=2$,
$$
P(u)\sim_{\mathbb{R}} \left\{\aligned
&(4-u)H-E+(u-1)R \ \text{ if } 0\leqslant u\leqslant 1, \\
&(4-u)H-E\ \text{ if } 1\leqslant u\leqslant 2,
\endaligned
\right.
$$
and
$$
N(u)= \left\{\aligned
&0\ \text{ if } 0\leqslant u\leqslant 1, \\
&(u-1)R\ \text{ if } 1\leqslant u\leqslant 2,
\endaligned
\right.
$$
which gives $S_{X}(S)=\frac{1}{22}\int\limits_{0}^{2}P(u)^3du=\frac{67}{88}$. Similarly, if $S\in|2H-E|$, then $\tau=\frac{3}{2}$,
$$
P(u)\sim_{\mathbb{R}} \left\{\aligned
&(4-2u)H+(u-1)E-R \ \text{ if } 0\leqslant u\leqslant 1, \\
&(4-2u)H-R\ \text{ if } 1\leqslant u\leqslant \frac{3}{2},
\endaligned
\right.
$$
and
$$
N(u)= \left\{\aligned
&0\ \text{ if } 0\leqslant u\leqslant 1, \\
&(u-1)E\ \text{ if } 1\leqslant u\leqslant \frac{3}{2},
\endaligned
\right.
$$
which gives $S_{X}(S)=\frac{109}{176}$.
Now, for every prime divisor $F$ over the surface $S$, we set
$$
S\big(W^S_{\bullet,\bullet};F\big)=\frac{3}{(-K_X)^3}\int\limits_0^{\tau}\mathrm{ord}_F\big(N(u)\vert_{S}\big)\big(P(u)\vert_{S}\big)^2du+\frac{3}{(-K_X)^3}\int\limits_0^\tau\int\limits_0^{\infty}\mathrm{vol}\big(P(u)\big\vert_{S}-vF\big)dvdu.
$$
Then, following \cite{AbbanZhuang,Book}, we let
$$
\delta_P\big(S,W^S_{\bullet,\bullet}\big)=\inf_{\substack{F/S\\P\in C_S(F)}}\frac{A_S(F)}{S\big(W^S_{\bullet,\bullet};F\big)},
$$
where the~infimum is taken by all prime divisors over the surface $S$ whose center on $S$ contains $P$.
Then it follows from \cite{AbbanZhuang,Book} that
\begin{equation}
\label{equation:AZ}\tag{$\bigstar$}
\frac{A_X(\mathbf{F})}{S_X(\mathbf{F})}\geqslant\min\Bigg\{\frac{1}{S_X(S)},\delta_P\big(S,W^S_{\bullet,\bullet}\big)\Bigg\}.
\end{equation}
Therefore, is $\beta(\mathbf{F})\leqslant 0$,
then $\delta_P(S,W^S_{\bullet,\bullet})\leqslant 1$.

\begin{remark}
\label{remark}
If $Z$ is a point and $\beta(\mathbf{F})\leqslant 0$, then it follows from \cite{AbbanZhuang,Book} that $\delta_P(S,W^S_{\bullet,\bullet})<1$.
\end{remark}

To estimate $\delta_P(S,W^S_{\bullet,\bullet})$, we set $D=P(u)\vert_S$. Then $D$ is ample for $u\in[0,\tau)$, and
$$
D^2=\left\{\aligned
&(2-u)(6-u) \ \text{ if } S\in|H-R|, \\
&2(3-2u)(5-2u) \ \text{ if } S\in|2H-E|.
\endaligned
\right.
$$
For $u\in[0,\tau)$, set
$$
\delta_P\big(S,D\big)=\inf_{\substack{F/S\\P\in C_S(F)}}\frac{A_S(F)}{S_D\big(F\big)},
$$
where
$$
S_D(F)=\frac{1}{D^2}\int\limits_0^{\infty}\mathrm{vol}\big(D-vF\big)dv,
$$
and the~infimum is taken by all prime divisors over  $S$ whose center on $S$ contains $P$.

\begin{lemma}[{cf. \cite[Lemma 27]{CheltsovFujitaKishimotoOkada}}]
\label{lemma:Nemuro}
Let $f\colon[1,\tau]\to\mathbb{R}_{>0}$ be a continuous positive function such that
$$
\delta_P\big(S,D\big)\geqslant f(u)
$$
for every $u\in[1,\tau)$. If $S\in|H-R|$ and $S\not\in R$, then
$$
\delta_P(S,W^S_{\bullet,\bullet})\geqslant\frac{1}{\frac{15}{22f(1)}+\frac{3}{22}\int\limits_1^2\frac{(2-u)(6-u)}{f(u)}du}.
$$
If $S\in|H-R|$ and $S\in R$, then
$$
\delta_P(S,W^S_{\bullet,\bullet})\geqslant\frac{1}{\frac{9}{88}+\frac{15}{22f(1)}+\frac{3}{22}\int\limits_1^2\frac{(2-u)(6-u)}{f(u)}du}.
$$
If $S\in|2H-E|$ and $S\not\in E$, then
$$
\delta_P(S,W^S_{\bullet,\bullet})\geqslant\frac{1}{\frac{9}{11f(1)}+\frac{3}{22}\int\limits_1^{\frac{3}{2}}\frac{2(3-2u)(5-2u)}{f(u)}du}.
$$
If $S\in|H-R|$ and $S\in E$, then
$$
\delta_P(S,W^S_{\bullet,\bullet})\geqslant\frac{1}{\frac{5}{176}+\frac{9}{11f(1)}+\frac{3}{22}\int\limits_1^{\frac{3}{2}}\frac{2(3-2u)(5-2u)}{f(u)}du}.
$$
\end{lemma}

\begin{proof}
The proof is the same as the proof of \cite[Lemma 27]{CheltsovFujitaKishimotoOkada}.
Namely, fix a prime divisor $F$ over $S$. Write $S\big(W^S_{\bullet,\bullet};F\big)=R_1+R_2+R_3$ for
\begin{align*}
R_1&=\frac{3}{22}\int\limits_0^1\int\limits_0^{\infty}\mathrm{vol}\big(P(u)\big\vert_{S}-vF\big)dvdu,\\
R_2&=\frac{3}{22}\int\limits_1^\tau\int\limits_0^{\infty}\mathrm{vol}\big(P(u)\big\vert_{S}-vF\big)dvdu,\\
R_3&=\frac{3}{22}\int\limits_1^{\tau}\mathrm{ord}_F\big(N(u)\vert_{S}\big)\big(P(u)\vert_{S}\big)^2du.
\end{align*}
Then
$$
R_1\leqslant \frac{3(-K_S)^2}{22\delta_P(S,-K_S)}A_S(F)\leqslant \frac{3(-K_S)^2}{22f(1)}A_S(F).
$$
Similarly, we see that
$$
R_2\leqslant A_S(F)\frac{3}{22}\int\limits_1^\tau\frac{D^2}{f(u)}du.
$$
Finally, observe that $\mathrm{ord}_F(N(u)\vert_{S})\leqslant (u-1)A_S(F)$,
since $\mathrm{Supp}(N(u)\vert_S)$ is a smooth curve contained in the smooth locus of the surface $S$ for $u\in(1,\tau]$.
Hence, we have
$$
R_3\leqslant
A_S(F)\frac{3}{22}\int\limits_1^{\tau}(u-1)\big(P(u)\vert_{S}\big)^2du=
\left\{\aligned
&\frac{9}{88}A_S(F) \ \text{ if } S\in|H-R|, \\
&\frac{5}{176}A_S(F) \ \text{ if } S\in|2H-E|.
\endaligned
\right.
$$
If $P\not\in\mathrm{Supp}(N(u))$ for $u\in(1,\tau]$, then~$R_3=0$.
Combining our estimates, we complete the proof.
\end{proof}

To use Lemma~\ref{lemma:Nemuro}, we must bound $\delta_P(S,D)$ for $u\in[1,\tau)$.
This is done in the next four~propositions.

\begin{proposition}
\label{proposition:dP5-smooth}
Suppose that $S\in|H-R|$, and the surface $S$ is smooth. Then
$$
\delta_P(S,D)\geqslant \frac{3(6-u)}{u^2-10u+22}
$$
for every $u\in[1,2)$.
\end{proposition}

\begin{proof}
Use\,Corollary~\ref{corollary:dP5}\,with\,$a=4-u$.\,If\,$u=1$,\,this\,follows\,from\,\cite[Proposition 2.7]{Akaike}\,or\,\mbox{\cite[\S~5.1]{Denisova}}.
\end{proof}

\begin{proposition}
\label{proposition:dP5-A1}
Let $S$ be the surface in $|H-R|$ containing $P$.
Suppose that $S$ has one singular~point, which is a singular~point of type $\mathbb{A}_1$.
Let $C$ be the fiber of the conic bundle $\phi\vert_S\colon S\to\mathbb{P}^1$ that contains the singular point of the surface $S$.
If $P\not\in C$ and $P\not\in E$, then
$$
\delta_P(S,D)\geqslant\frac{3(6-u)}{u^2-10u+22}
$$
for every $u\in[1,2)$. Similarly, if $P\not\in C$ and $P\in E$, then
$$
\delta_P(S,D)\geqslant
\left\{\aligned
&\frac{1}{2-u} \ \text{ if } 1\leqslant u\leqslant \frac{7-\sqrt{21}}{2}, \\
&\frac{3(6-u)}{u^2-10u+22}\ \text{ if } \frac{7-\sqrt{21}}{2}\leqslant u<2.
\endaligned
\right.
$$
\end{proposition}

\begin{proof}
Use\,Corollary\,\ref{corollary:dP5-A1}\,with\,$a=4-u$.\,If\,$u=1$,\,this\,follows\,from\,\cite[Proposition 2.1]{Akaike}\,or\,\cite[\S~5.2]{Denisova}.
\end{proof}

\begin{proposition}
\label{proposition:dP5-A2}
Let $S$ be the surface in $|H-R|$ containing $P$.
Suppose that $S$ has one singular~point, which is a singular point of type $\mathbb{A}_2$.
Let $C$ be the fiber of the conic bundle $\phi\vert_S\colon S\to\mathbb{P}^1$ that contains the singular point of the surface $S$.
If $P\not\in C$ and $P\not\in E$, then
$$
\delta_P(S,D)\geqslant
\left\{\aligned
&\frac{6(6-u)}{(2-u)(22+u)}\ \text{ if } 1\leqslant u\leqslant \frac{1+\sqrt{21}}{5}, \\
&\frac{4(6-u)}{u^2-14u+28}\ \text{ if } \frac{1+\sqrt{21}}{5}\leqslant u\leqslant \sqrt{5}-1, \\
&\frac{2(6-u)}{u^2-6u+12}\ \text{ if } \sqrt{5}-1\leqslant u<2.
\endaligned
\right.
$$
Similarly, if $P\not\in C$ and $P\in E$, then
$$
\delta_P(S,D)\geqslant
\left\{\aligned
&\frac{6(6-u)}{(2-u)(38-7u)} \ \text{ if } 1\leqslant u\leqslant \frac{7-\sqrt{17}}{2}, \\
&\frac{6(6-u)}{u^2-10u+28}\ \text{ if } \frac{7-\sqrt{17}}{2}\leqslant u<2.
\endaligned
\right.
$$
\end{proposition}

\begin{proof}
Use\,Corollary\,\ref{corollary:dP5-A2}\,with\,$a=4-u$.\,If\,$u=1$,\,this\,follows\,from\,\cite[Proposition 2.4]{Akaike}\,or\,\cite[\S~5.6]{Denisova}.
\end{proof}

\begin{proposition}
\label{proposition:dP6-smooth}
Let $S$ be the surface in  $|2H-E|$ containing $P$.
Suppose that $S$~is~smooth. Then
$$
\delta_P(S,D)\geqslant
\left\{\aligned
&\frac{1}{2-u} \ \text{ if $P$ is contained in a $(-1)$-curve and $1\leqslant u<\frac{3}{2}$}, \\
&\frac{2(5-2u)}{4u^2-18u+19}\ \text{ if $P$ is not contained in a $(-1)$-curve and $1\leqslant u\leqslant \frac{9-\sqrt{21}}{4}$},\\
&\frac{3(5-2u)}{4u^2-18u+21}\ \text{ if $P$ is not contained in a $(-1)$-curve and $\frac{9-\sqrt{21}}{4}\leqslant u<\frac{3}{2}$}.
\endaligned
\right.
$$
\end{proposition}

\begin{proof}
Apply Lemmas~\ref{lemma:dP6-e1-e2}, \ref{lemma:dP6-lines} and \ref{lemma:dP6-general-point} with $a=4-2u$.
If $u=1$, the required inequality follows from \cite[Proposition 3.6]{Akaike} or \cite[\S~4.1]{Denisova}.
\end{proof}

Now, applying Lemma~\ref{lemma:Nemuro} together with Proposition~\ref{proposition:dP5-smooth}, we get

\begin{corollary}
\label{corollary:dP5-smooth}
Let $S$ be the surface in $|H-R|$ such that $P\in S$. If $S$ is smooth, then $\delta_P(S,W^S_{\bullet,\bullet})>1$.
\end{corollary}

Similarly, applying Lemma~\ref{lemma:Nemuro} together with Propositions~\ref{proposition:dP5-A1} and \ref{proposition:dP5-A2}, we get

\begin{corollary}
\label{corollary:dP5-singular}
Let $S$ be the surface in $|H-R|$ that contains $P$. Suppose that $S$ has one singular~point.
Let $C$ be the fiber of the conic bundle $\phi\vert_S\colon S\to\mathbb{P}^1$ that contains the singular point of the surface $S$.
Suppose that $P\not\in C$. Then $\delta_P\big(S,W^S_{\bullet,\bullet}\big)>1$.
\end{corollary}

Finally, applying  Lemma~\ref{lemma:Nemuro} together with Proposition~\ref{proposition:dP6-smooth}, we get

\begin{corollary}
\label{corollary:dP6-smooth}
Let $S$ be the surface in  $|2H-E|$ that contains $P$.
Suppose that $S$~is~smooth. Then
$$
\delta_P\big(S,W^S_{\bullet,\bullet}\big)\geqslant 1.
$$
Moreover, if $P\not\in E$ or $P$ is not contained in a $(-1)$-curve in $S$, then $\delta_P\big(S,W^S_{\bullet,\bullet}\big)>1$.
\end{corollary}

Now, we are ready to show that $\beta(\mathbf{F})>0$.
Suppose that $\beta(\mathbf{F})\leqslant 0$. Let us seek for a contradiction.
We may assume that $P$ is a general point in $Z$.
Let $S$ be the surface in  $|H-R|$ that contains~$P$.
If~$S$~is smooth, $\delta_P(S,W^S_{\bullet,\bullet})>1$ by Corollary~\ref{corollary:dP5-smooth}, which contradicts \eqref{equation:AZ}.
Therefore, the surface $S$~is~singular. This implies that $Z\subset S$.

Recall that $S$ has one singular point, and this singular point is either a singular point of type $\mathbb{A}_1$
or a singular point of type $\mathbb{A}_2$. Set $O=\mathrm{Sing}(S)$,
let $S^\prime$ be the surface in  $|2H-E|$ that contains~$O$, and set $C=S\cap S^\prime$.
Then $C$ is a fiber of the conic bundle $\phi\vert_S\colon S\to\mathbb{P}^1$.
Moreover, we have
$$
\boxed{Z\subset C}
$$
since otherwise $\delta_P(S,W^S_{\bullet,\bullet})>1$ by Corollary~\ref{corollary:dP5-singular}, which is impossible by \eqref{equation:AZ}.

We claim that $S^\prime$ is smooth. Indeed, it follows from \cite{CorayTsfasman} that $C=\ell_1+\ell_2$,
where $\ell_1$ and $\ell_2$ are~smooth rational curves that intersects transversally at $O$.
Note that
\begin{itemize}
\item $\pi(S)$ is a plane such that $L\subset\pi(S)$,
\item $\pi(S^\prime)$ is a quadric surface such that $C_4\subset\pi(S^\prime)$.
\item $\pi(\ell_1)$ and $\pi(\ell_2)$ are lines such that $\pi(O)=\pi(\ell_1)\cap\pi(\ell_2)$.
\end{itemize}
Since the surface $S$ is singular at $O$, the plane $\pi(S)$ is tangent to $C_4$ at the point $\pi(O)$.
In particular, we see that $\pi(O)\in C_4$, which implies that $S^\prime$ is smooth at $\pi(O)$,
because $C_4$ is smooth. Since
$$
\pi(S)\cap\pi(S^\prime)=\pi(\ell_1)\cup\pi(\ell_2),
$$
we conclude that the quadric surface $\pi(S^\prime)$ is smooth. Moreover, since
$$
L\cap \pi(S^\prime)=\big(L\cap\pi(\ell_1)\big)\cup\big(L\cap\pi(\ell_2)\big),
$$
we see that the line $L$ intersects the quadric $\pi(S^\prime)$ transversally by the points $L\cap\pi(\ell_1)$ and $L\cap\pi(\ell_2)$,
which implies that the surface $S^\prime$ is smooth.

We see that $S^\prime$ is a smooth sextic del Pezzo surface.
Observe that $E\cap S^\prime$ is a smooth elliptic~curve.
Moreover, applying Corollary~\ref{corollary:dP6-smooth} and \eqref{equation:AZ} to the surface $S^\prime$,
we see that $P$ is one of finitely many intersection points of this elliptic curve with six $(-1)$-curves in $S^\prime$.
This shows that $Z$ is a point, so~that $Z=P$. 
Since $\delta_P(S^\prime,W^{S^\prime}_{\bullet,\bullet})\geqslant 1$ by Corollary~\ref{corollary:dP6-smooth},
it follows from \eqref{equation:AZ} that
$$
1\geqslant\frac{A_X(\mathbf{F})}{S_X(\mathbf{F})}\geqslant\min\Bigg\{\frac{1}{S_X(S^\prime)},\delta_P\big(S,W^{S^\prime}_{\bullet,\bullet}\big)\Bigg\}=
\min\Bigg\{\frac{176}{109},\delta_P\big(S^\prime,W^{S^\prime}_{\bullet,\bullet}\big)\Bigg\}=\min\Bigg\{\frac{176}{109},1\Bigg\}=1.
$$
Now, using Remark~\ref{remark}, we obtain a contradiction, because $Z$ is a point in $S^\prime$.
Theorem~\ref{theorem:main} is proved.

\appendix

\section{$\delta$-invariants of some polarized del Pezzo surfaces}
\label{section:delta}

Let $S$ be a del Pezzo surface with at most Du Val singularities, let $D$ be an ample divisor on $S$.
For every prime divisor $F$ over $S$, set
$$
S_D(F)=\frac{1}{D^2}\int\limits_0^{\infty}\mathrm{vol}\big(D-vF\big)dv.
$$
Let $P$ be a smooth point in $S$, and let
$$
\delta_P\big(S,D\big)=\inf_{\substack{F/S\\P\in C_S(F)}}\frac{A_S(F)}{S_D\big(F\big)},
$$
where the~infimum is taken by all prime divisors over $S$ whose center on $S$ contains $P$. Set
$$
\delta\big(S,D\big)=\inf_{P\in S}\delta_P\big(S,D\big).
$$
In~this appendix, we estimate $\delta_P(S,D)$ in some cases similar to what is done in \cite{Akaike,Denisova} for $D=-K_S$.

To explain how to estimate  $\delta_P(S,D)$, fix a smooth curve $C\subset S$ that passes through $P$. Set
$$
\tau=\mathrm{sup}\Big\{u\in\mathbb{R}_{\geqslant 0}\ \big\vert\ \text{the divisor  $D-vC$ is pseudo-effective}\Big\}.
$$
For~$v\in[0,\tau]$, let $P(v)$ be the~positive part of the~Zariski decomposition of the~divisor $D-vC$,
and let $N(v)$ be its negative part. Then
$$
S_D(C)=\frac{1}{D^2}\int\limits_{0}^{\infty}\mathrm{vol}\big(D-vC\big)dv=\frac{1}{D^2}\int\limits_{0}^{\tau}P(v)^2dv.
$$
Note that $\delta_P(S,D)\leqslant\frac{1}{S_D(C)}$, since $A_S(C)=1$.
To estimate $\delta_P(S,D)$ from below, set
$$
S\big(W^C_{\bullet,\bullet};P\big)=\frac{2}{D^2}\int\limits_0^{\tau}\mathrm{ord}_P\big(N(v)\vert_{C}\big)\big(P(v)\cdot C\big)dv+\frac{1}{D^2}\int\limits_0^\tau\big(P(v)\cdot C\big)^2dv.
$$
Then it follows from \cite{AbbanZhuang,Book} that
\begin{equation}
\label{equation:AZ-surface}\tag{$\heartsuit$}
\delta_P\big(S,D\big)\geqslant\min\Bigg\{\frac{1}{S_D(C)},\frac{1}{S\big(W^C_{\bullet,\bullet};P\big)}\Bigg\}.
\end{equation}
Usually, \eqref{equation:AZ-surface} gives a very good estimate for $\delta_P(S,D)$ when $C^2<0$.
If $P$ is not contained in any curve with negative self-intersection, we have to blow up the surface $S$ at the point $P$,
and apply similar arguments to the exceptional curve of this blow up.

Namely, let $f\colon\widetilde{S}\to S$ be the blow up of $S$ at the point $P$, and let $E$ be the $f$-exceptional curve.
In all applications, the surface $\widetilde{S}$ will be a del Pezzo surface with at most Du Val singularities.~Set
$$
\widetilde{\tau}=\mathrm{sup}\Big\{u\in\mathbb{R}_{\geqslant 0}\ \big\vert\ \text{the divisor  $f^*(D)-vE$ is pseudo-effective}\Big\}.
$$
For~$v\in[0,\widetilde{\tau}]$, let $\widetilde{P}(v)$ be the~positive part of the~Zariski decomposition of the~divisor $f^*(D)-vE$,
and let $\widetilde{N}(v)$ be its negative part. Then
$$
S_D(E)=\frac{1}{D^2}\int\limits_{0}^{\infty}\mathrm{vol}\big(f^*(D)-vC\big)dv=\frac{1}{D^2}\int\limits_{0}^{\widetilde{\tau}}\widetilde{P}(v)^2dv.
$$
Note that $\delta_P(S,D)\leqslant\frac{2}{S_D(E)}$, since $A_S(E)=2$.
Now, for every point $O\in E$, we set
$$
S\big(W^E_{\bullet,\bullet};O\big)=\frac{2}{D^2}\int\limits_0^{\widetilde{\tau}}\mathrm{ord}_O\big(\widetilde{N}(v)\vert_{E}\big)\big(\widetilde{P}(v)\vert_{E}\big)dv+\frac{1}{D^2}\int\limits_0^{\widetilde{\tau}}\big(\widetilde{P}(v)\cdot E\big)^2dv.
$$
Then it follows from \cite{AbbanZhuang,Book} that
\begin{equation}
\label{equation:AZ-surface-blow-up}\tag{$\diamondsuit$}
\delta_P\big(S,D\big)\geqslant\min\Bigg\{\frac{2}{S_D(E)},\inf_{O\in E}\frac{1}{S\big(W^E_{\bullet,\bullet};O\big)}\Bigg\}.
\end{equation}
In the next four subsections, we will apply \eqref{equation:AZ-surface} and \eqref{equation:AZ-surface-blow-up} using notations introduced here.

\subsection{Smooth quintic del Pezzo surface}
\label{subsection:dP5-smooth}

Let $S$ be a smooth del Pezzo surface such that $K_S^2=5$.
There is a birational morphism $\pi\colon S\to\mathbb{P}^2$ that blows up $4$ points.
Let $\mathbf{e}_1$, $\mathbf{e}_2$, $\mathbf{e}_3$, $\mathbf{e}_4$ be the~exceptional curves of the morphism $\pi$,
and let $\mathbf{h}=\pi^*(\mathcal{O}_{\mathbb{P}^2}(1))$. Set
$$
D=a\mathbf{h}-\mathbf{e}_1-\mathbf{e}_2-\mathbf{e}_3-\mathbf{e}_4
$$
for $a\in(2,3]$. Then $D$ is ample and $D^2=a^2-4$.

\begin{lemma}
\label{lemma:dP5-e1-e2-e3-e4}
Let $P$ be a point in $\mathbf{e}_1\cup\mathbf{e}_2\cup\mathbf{e}_3\cup\mathbf{e}_4$.
Then
$$
\delta_P\big(S,D\big)\geqslant\frac{3(a + 2)}{a^2+2a-2}.
$$
\end{lemma}

\begin{proof}
We may assume that $P\in\mathbf{e}_1$. Set $C=\mathbf{e}_1$. Then $\tau=2a-4$. Moreover, we have
$$
P(v)\sim_{\mathbb{R}} \left\{\aligned
&a\mathbf{h}-(1+v)\mathbf{e}_1-\mathbf{e}_2-\mathbf{e}_3-\mathbf{e}_4  \ \text{ if } 0\leqslant v\leqslant a-2, \\
&(4a-3v-6)\mathbf{h}+(2v+5-3a)\mathbf{e}_1+(1-a+v)\big(\mathbf{e}_2+\mathbf{e}_3+\mathbf{e}_4\big)\ \text{ if } a-2\leqslant v\leqslant 2a-4,
\endaligned
\right.
$$
and
$$
N(v)= \left\{\aligned
&0\ \text{ if } 0\leqslant v\leqslant a-2, \\
&(v+2-a)\big(\mathbf{l}_{12}+\mathbf{l}_{13}+\mathbf{l}_{14}\big)\ \text{ if } a-2\leqslant v\leqslant 2a-4,
\endaligned
\right.
$$
where $\mathbf{l}_{12}$, $\mathbf{l}_{13}$, $\mathbf{l}_{14}$ are $(-1)$-curves in $|\mathbf{h}-\mathbf{e}_1-\mathbf{e}_2|$,
$|\mathbf{h}-\mathbf{e}_1-\mathbf{e}_3|$, $|\mathbf{h}-\mathbf{e}_1-\mathbf{e}_4|$, respectively. Then
$$
P(v)^2=\left\{\aligned
&a^2-v^2-2v-4  \ \text{ if } 0\leqslant v\leqslant a-2, \\
&2(2a-v-4)(a-v-1)\ \text{ if } a-2\leqslant v\leqslant 2a-4,
\endaligned
\right.
$$
and
$$
P(v)\cdot C=\left\{\aligned
&1+v\ \text{ if } 0\leqslant v\leqslant a-2, \\
&3a-2v-5\ \text{ if } a-2\leqslant v\leqslant 2a-4.
\endaligned
\right.
$$
Integrating, we get $S_D(C)=\frac{(a-2)(a+10)}{3(a+2)}$ and
$$
S\big(W^C_{\bullet,\bullet};P\big)=\frac{2}{a^2-4}\int\limits_{a-2}^{2a-4}\mathrm{ord}_P\big(N(v)\vert_{C}\big)\big(P(v)\cdot C\big)dv+\frac{2a^2-5a+8}{3(a+2)}.
$$
Thus, if $P\not\in\mathbf{l}_{12}\cup\mathbf{l}_{13}\cup\mathbf{l}_{14}$, then $S(W^C_{\bullet,\bullet};P)=\frac{2a^2-5a+8}{3(a+2)}$.
Similarly, if $P\in\mathbf{l}_{12}\cup\mathbf{l}_{13}\cup\mathbf{l}_{14}$, then
$$
S\big(W^C_{\bullet,\bullet};P\big)=\frac{2}{a^2-4}\int\limits_{a-2}^{2a-4}(v+2-a)\big(P(v)\cdot C\big)dv+\frac{2a^2-5a+8}{3a+6}=\frac{a^2+2a-2}{3(a+2)}.
$$
Now, using \eqref{equation:AZ-surface}, we obtain the required assertion.
\end{proof}

As in the proof of Lemma~\ref{lemma:dP5-e1-e2-e3-e4}, let $\mathbf{l}_{12}$, $\mathbf{l}_{13}$, $\mathbf{l}_{14}$
$\mathbf{l}_{23}$, $\mathbf{l}_{24}$, $\mathbf{l}_{34}$ be $(-1)$-curves in $|\mathbf{h}-\mathbf{e}_1-\mathbf{e}_2|$,
$|\mathbf{h}-\mathbf{e}_1-\mathbf{e}_3|$, $|\mathbf{h}-\mathbf{e}_1-\mathbf{e}_4|$,
$|\mathbf{h}-\mathbf{e}_2-\mathbf{e}_3|$, $|\mathbf{h}-\mathbf{e}_2-\mathbf{e}_4|$,
$|\mathbf{h}-\mathbf{e}_3-\mathbf{e}_4|$, respectively.

\begin{lemma}
\label{lemma:dP5-lines}
Let $P$ be a point in $\mathbf{l}_{12}\cup\mathbf{l}_{13}\cup\mathbf{l}_{14}\cup\mathbf{l}_{23}\cup\mathbf{l}_{24}\cup\mathbf{l}_{34}$.
Then
$$
\delta_P\big(S,D\big)\geqslant\frac{3(a + 2)}{a^2+2a-2}.
$$
\end{lemma}

\begin{proof}
We may assume that $P\in\mathbf{l}_{12}$.
By Lemma~\ref{lemma:dP5-e1-e2-e3-e4}, we may also assume that $P\not\in\mathbf{e}_1\cup\mathbf{e}_2\cup\mathbf{e}_3\cup\mathbf{e}_4$.
Set $C=\mathbf{l}_{12}$. Then $\tau=a-1$. Moreover, we have
$$
P(v)\sim_{\mathbb{R}} \left\{\aligned
&(a-v)\mathbf{h}-(1-v)\big(\mathbf{e}_1+\mathbf{e}_2\big)-\mathbf{e}_3-\mathbf{e}_4  \ \text{ if } 0\leqslant v\leqslant a-2, \\
&(2a-2v-2)\mathbf{h}+(v-1)\big(\mathbf{e}_1+\mathbf{e}_2\big)+(1-a+v)\big(\mathbf{e}_3+\mathbf{e}_4\big)\ \text{ if } a-2\leqslant v\leqslant 1, \\
&(2a-2v-2)\mathbf{h}+(1-a+v)\big(\mathbf{e}_3+\mathbf{e}_4\big)\ \text{ if } 1\leqslant v\leqslant a-1,
\endaligned
\right.
$$
and
$$
N(v)=\left\{\aligned
&0  \ \text{ if } 0\leqslant v\leqslant a-2, \\
&(v+2-a)\mathbf{l}_{34} \ \text{ if } a-2\leqslant v\leqslant 1, \\
&(v+2-a)\mathbf{l}_{34}+(v-1)\big(\mathbf{e}_1+\mathbf{e}_2\big)\ \text{ if } 1\leqslant v\leqslant a-1.
\endaligned
\right.
$$
This gives
$$
P(v)^2=\left\{\aligned
&a^2-2av-v^2+4v-4  \ \text{ if } 0\leqslant v\leqslant a-2, \\
&2(a-2)(a-2v) \ \text{ if } a-2\leqslant v\leqslant 1, \\
&2(a-v-1)^2\ \text{ if } 1\leqslant v\leqslant a-1,
\endaligned,
\right.
$$
and
$$
P(v)\cdot C=\left\{\aligned
&a+v-2  \ \text{ if } 0\leqslant v\leqslant a-2, \\
&2a-4 \ \text{ if } a-2\leqslant v\leqslant 1, \\
&2a-2v-2\ \text{ if } 1\leqslant v\leqslant a-1.
\endaligned,
\right.
$$
Now, integrating, we get $S_D(C)=\frac{a^2+2a-2}{3(a+2)}$, which gives $\delta_P(S,D)\leqslant\frac{3(a+2)}{a^2+2a-2}$.
Similarly, we get
$$
S\big(W^C_{\bullet,\bullet};P\big)=
\frac{2}{a^2-4}\int\limits_{a-2}^{a-1}\mathrm{ord}_P\big(N(v)\vert_{C}\big)\big(P(v)\cdot C\big)dv+\frac{(a-2)(14-a)}{3(a+2)}.
$$
Thus, if $P\not\in\mathbf{l}_{34}$, then $S(W^C_{\bullet,\bullet};P)=\frac{(a-2)(14-a)}{3(a+2)}$.
Similarly, if $P\in\mathbf{l}_{34}$, then
$$
S\big(W^C_{\bullet,\bullet};P\big)=\frac{2}{a^2-4}\int\limits_{a-2}^{a-1}(v+2-a)\big(P(v)\cdot C\big)dv+\frac{(a-2)(14-a)}{3(a+2)}=\frac{a^2+2a-2}{3(a+2)}=S_D(C),
$$
so that $\delta_P(S,D)\geqslant\frac{3(a + 2)}{a^2+2a-2}$ by \eqref{equation:AZ-surface}.
\end{proof}

Finally, we prove

\begin{lemma}
\label{lemma:dP5-general-point}
Let $P$ be a point in $S\setminus(\mathbf{e}_1\cup\mathbf{e}_2\cup\mathbf{e}_3\cup\mathbf{e}_4\cup\mathbf{l}_{12}\cup\mathbf{l}_{13}\cup\mathbf{l}_{14}\cup\mathbf{l}_{23}\cup\mathbf{l}_{24}\cup\mathbf{l}_{34})$.
Then
$$
\delta_P\big(S,D\big)\geqslant
\left\{\aligned
&\frac{2(a+2)}{a^2-2a+4} \ \text{ if } 2<a\leqslant 5-\sqrt{5}, \\
&\frac{2(2a+4)}{a^2+6a-12}\ \text{ if }  5-\sqrt{5}\leqslant a\leqslant 3.
\endaligned
\right.
$$
\end{lemma}

\begin{proof}
Recall that $f\colon\widetilde{S}\to S$ is a blow up of $S$ at the point $P$,
and $E$ is the $f$-exceptional curve.
Note that $\widetilde{S}$ is a del Pezzo surface of degree $4$.
Let
$\widetilde{\mathbf{e}}_1$, $\widetilde{\mathbf{e}}_2$, $\widetilde{\mathbf{e}}_3$, $\widetilde{\mathbf{e}}_4$, $\widetilde{\mathbf{l}}_{12}$, $\widetilde{\mathbf{l}}_{13}$, $\widetilde{\mathbf{l}}_{14}$, $\widetilde{\mathbf{l}}_{23}$, $\widetilde{\mathbf{l}}_{24}$, $\widetilde{\mathbf{l}}_{34}$
be the strict transforms on $\widetilde{S}$ of the $(-1)$-curves
$\mathbf{e}_1$, $\mathbf{e}_2$, $\mathbf{e}_3$, $\mathbf{e}_4$, $\mathbf{l}_{12}$, $\mathbf{l}_{13}$, $\mathbf{l}_{14}$, $\mathbf{l}_{23}$, $\mathbf{l}_{24}$, $\mathbf{l}_{34}$,
respectively.
Set $\widetilde{\mathbf{h}}=f^*(\mathbf{h})$.
Let $\mathbf{c}_0$, $\mathbf{c}_1$, $\mathbf{c}_2$, $\mathbf{c}_3$, $\mathbf{c}_4$
be the $(-1)$-curves in $|2\widetilde{\mathbf{h}}-E-\widetilde{\mathbf{e}}_1-\widetilde{\mathbf{e}}_2-\widetilde{\mathbf{e}}_3-\widetilde{\mathbf{e}}_4|$,
$|\widetilde{\mathbf{h}}-E-\widetilde{\mathbf{e}}_1|$, $|\widetilde{\mathbf{h}}-E-\widetilde{\mathbf{e}}_2|$,
$|\widetilde{\mathbf{h}}-E-\widetilde{\mathbf{e}}_3|$, $|\widetilde{\mathbf{h}}-E-\widetilde{\mathbf{e}}_4|$, respectively.
Then
\begin{center}
$\mathbf{c}_0$, $\mathbf{c}_1$, $\mathbf{c}_2$, $\mathbf{c}_3$, $\mathbf{c}_4$, $\widetilde{\mathbf{e}}_1$, $\widetilde{\mathbf{e}}_2$, $\widetilde{\mathbf{e}}_3$, $\widetilde{\mathbf{e}}_4$, $\widetilde{\mathbf{l}}_{12}$, $\widetilde{\mathbf{l}}_{13}$, $\widetilde{\mathbf{l}}_{14}$, $\widetilde{\mathbf{l}}_{23}$, $\widetilde{\mathbf{l}}_{24}$, $\widetilde{\mathbf{l}}_{34}$, $E$
\end{center}
are all $(-1)$-curves in $\widetilde{S}$. These curves generates the Mori cone of the surface $\widetilde{S}$.

We compute $\widetilde{\tau}=\frac{3a-4}{2}$. Similarly, we see that
$$
\widetilde{P}(v)\sim_{\mathbb{R}} \left\{\aligned
&a\widetilde{\mathbf{h}}-\widetilde{\mathbf{e}}_1-\widetilde{\mathbf{e}}_2-\widetilde{\mathbf{e}}_3-\widetilde{\mathbf{e}}_4-vE \ \text{ if } 0\leqslant v\leqslant 2a-4, \\
&(5a-2v-8)\widetilde{\mathbf{h}}+(3-2a+v)\big(\widetilde{\mathbf{e}}_1+\widetilde{\mathbf{e}}_2+\widetilde{\mathbf{e}}_3+\widetilde{\mathbf{e}}_4\big)-(2a-4)E\ \text{ if } 2a-4\leqslant v\leqslant a-1, \\
&(3a-2v-4)\big(3\widetilde{\mathbf{h}}-\widetilde{\mathbf{e}}_1-\widetilde{\mathbf{e}}_2-\widetilde{\mathbf{e}}_3-\widetilde{\mathbf{e}}_4-2E\big) \ \text{ if } a-1\leqslant v\leqslant \frac{3a-4}{2},
\endaligned
\right.
$$
and
$$
\widetilde{N}(v)=\left\{\aligned
&0\ \text{ if } 0\leqslant v\leqslant 2a-4, \\
&(v+4-2a)\mathbf{c}_0 \ \text{ if } 2a-4\leqslant v\leqslant a-1, \\
&(v+4-2a)\mathbf{c}_0+(v+1-a)\big(\mathbf{c}_1+\mathbf{c}_2+\mathbf{c}_3+\mathbf{c}_4\big) \ \text{ if } a-1\leqslant v\leqslant \frac{3a-4}{2}.
\endaligned
\right.
$$
This gives
$$
\widetilde{P}(v)^2=\left\{\aligned
&a^2-v^2-4\ \text{ if } 0\leqslant v\leqslant 2a-4, \\
&(a-2)(5a-4v-6) \ \text{ if } 2a-4\leqslant v\leqslant a-1, \\
&(3a-2v-4)^2 \ \text{ if } a-1\leqslant v\leqslant \frac{3a-4}{2},
\endaligned
\right.
$$
and
$$
\widetilde{P}(v)\cdot E=\left\{\aligned
&v \ \text{ if } 0\leqslant v\leqslant 2a-4, \\
&2a-4 \ \text{ if } 2a-4\leqslant v\leqslant a-1, \\
&6a-4v-8\ \text{ if } a-1\leqslant v\leqslant \frac{3a-4}{2}.
\endaligned
\right.
$$
Now, integrating, we get $S_D(E)=\frac{a^2+6a-12}{2(a+2)}$.

Let $O$ be a point in $E$. Then
$$
S\big(W^E_{\bullet,\bullet};O\big)=
\frac{2}{a^2-4}\int\limits_{2a-4}^{\frac{3a-4}{2}}\mathrm{ord}_O\big(\widetilde{N}(v)\vert_{E}\big)\big(\widetilde{P}(v)\cdot E\big)dv+\frac{2(8-a)(a-2)}{3(a+2)}.
$$
Thus, if $O\not\in\mathbf{c}_0\cup\mathbf{c}_1\cup\mathbf{c}_2\cup\mathbf{c}_3\cup\mathbf{c}_4$, then $S(W^E_{\bullet,\bullet};O)=\frac{2(8-a)(a-2)}{3(a+2)}$.
Similarly, if $O\in\mathbf{c}_0$, then
$$
S\big(W^E_{\bullet,\bullet};O\big)=
\frac{2}{a^2-4}\int\limits_{2a-4}^{\frac{3a-4}{2}}(v+4-2a)\big(\widetilde{P}(v)\cdot E\big)dv+\frac{2(8-a)(a-2)}{3(a+2)}=\frac{a^2-2a+4}{2(a+2)}.
$$
Likewise, if $O\in\mathbf{c}_1\cup\mathbf{c}_2\cup\mathbf{c}_3\cup\mathbf{c}_4$, then
$$
S\big(W^E_{\bullet,\bullet};O\big)=
\frac{2}{a^2-4}\int\limits_{a-1}^{\frac{3a-4}{2}}(v + 1 - a)\big(\widetilde{P}(v)\cdot E\big)dv+\frac{2(8-a)(a-2)}{3(a+2)}=\frac{(a-2)(10-a)}{2(a+2)}.
$$
Since $\frac{a^2-2a+4}{2(a+2)}\geqslant \frac{(a-2)(10-a)}{2(a+2)}$ for $a\in(2,3]$, we have 
$$
\inf_{O\in E}\frac{1}{S\big(W^E_{\bullet,\bullet};O\big)}=\frac{2(a+2)}{a^2-2a+4}.
$$
Therefore, it follows from \eqref{equation:AZ-surface-blow-up} that
$$
\delta_P(S,D)\geqslant\min\Bigg\{\frac{2(2a+4)}{a^2+6a-12},\frac{2(a+2)}{a^2-2a+4}\Bigg\},
$$
which gives the required assertion.
\end{proof}

Combining Lemmas~\ref{lemma:dP5-e1-e2-e3-e4}, \ref{lemma:dP5-lines}, \ref{lemma:dP5-general-point}, we obtain

\begin{corollary}
\label{corollary:dP5}
One has $\delta(S,D)\geqslant\frac{3(a + 2)}{a^2+2a-2}$ for every $a\in(2,3]$.
\end{corollary}

In fact, the proofs of Lemmas~\ref{lemma:dP5-e1-e2-e3-e4}, \ref{lemma:dP5-lines}, \ref{lemma:dP5-general-point} give $\delta(S,D)=\frac{3(a + 2)}{a^2+2a-2}$ for every $a\in(2,3]$.

\subsection{Quintic del Pezzo surface with singular point of type $\mathbb{A}_1$}
\label{subsection:dP5-A1}

Let $S$ be a del Pezzo~surface such that $K_S^2=5$, and $S$ has one singular point,
which is a singular point of type $\mathbb{A}_1$. Set~$O=\mathrm{Sing}(S)$.
It follows from \cite{CorayTsfasman} that there exists a birational morphism $\pi\colon S\to\mathbb{P}^2$ that contracts
three smooth irreducible rational curves
$\mathbf{e}_1$, $\mathbf{e}_2$, $\mathbf{e}_3$ such that $\mathbf{e}_1$ and $\mathbf{e}_2$ are $(-1)$-curves
contained in the smooth locus of the del Pezzo surface $S$, the curve $\mathbf{e}_3$ contains the point $O$,~and~$\mathbf{e}_3^2=-\frac{1}{2}$.
Set~$\mathbf{h}=\pi^*(\mathcal{O}_{\mathbb{P}^2}(1))$.
Let $\mathbf{l}_1$, $\mathbf{l}_2$, $\mathbf{l}_3$, $\mathbf{l}_4$ be irreducible curves in
$|\mathbf{h}-\mathbf{e}_1-\mathbf{e}_3|$, $|\mathbf{h}-\mathbf{e}_2-\mathbf{e}_3|$,
$|\mathbf{h}-2\mathbf{e}_3|$, $|\mathbf{h}-\mathbf{e}_1-\mathbf{e}_2|$, respectively.
Observe that $O=\mathbf{l}_1\cap\mathbf{l}_2\cap\mathbf{e}_3$, and $\mathbf{e}_1$, $\mathbf{e}_2$, $\mathbf{e}_3$, $\mathbf{l}_1$, $\mathbf{l}_2$, $\mathbf{l}_3$, $\mathbf{l}_4$
are all curves in $S$ that have negative self-intersections. 
These curves are smooth, and their intersections are given in the following table:

\begin{center}
\renewcommand\arraystretch{1.5}
\begin{tabular}{|c||c|c|c|c|c|c|c|}
\hline
$\bullet$      & $\mathbf{e}_1$ & $\mathbf{e}_2$ & $\mathbf{e}_3$ & $\mathbf{l}_1$ & $\mathbf{l}_2$ & $\mathbf{l}_3$ & $\mathbf{l}_4$\\
\hline\hline
$\mathbf{e}_1$ &    $-1$        &     $0$        &      $0$       &    $1$         &    $0$         &       $0$      & $1$        \\
\hline
$\mathbf{e}_2$ &     $0$        &     $-1$        &      $0$       &    $0$         &    $1$         &       $0$      &  $1$       \\
\hline
$\mathbf{e}_3$ &     $0$        &     $0$        &  $-\frac{1}{2}$&   $\frac{1}{2}$&  $\frac{1}{2}$ &      $1$       &  $0$     \\
\hline
$\mathbf{l}_1$ &     $1$        &     $0$        &   $\frac{1}{2}$&  $-\frac{1}{2}$&  $\frac{1}{2}$ &    $0$         &  $0$     \\
\hline
$\mathbf{l}_2$ &     $0$        &     $1$        &   $\frac{1}{2}$& $\frac{1}{2}$  & $-\frac{1}{2}$ &     $0$        &  $0$    \\
\hline
$\mathbf{l}_3$ &     $0$        &     $0$        &  $1$           &   $0$          &     $0$        &  $-1$          &   $1$  \\
\hline
$\mathbf{l}_4$ &     $1$        &     $1$        &   $0$          &   $0$          &     $0$        &    $1$         & $-1$   \\
\hline
\end{tabular}
\end{center}

Set $D=a\mathbf{h}-\mathbf{e}_1-\mathbf{e}_2-2\mathbf{e}_3$ for $a\in(2,3]$. Then $D$ is ample and $D^2=a^2-4$.

\begin{lemma}
\label{lemma:dP5-A1-e3}
Let $P$ be a point in $\mathbf{e}_3$. Then
$$
\delta_P\big(S,D\big)\geqslant
\left\{\aligned
&\frac{3(a+2)}{a^2+2a-2} \ \text{ if } 2<a\leqslant \frac{1+\sqrt{21}}{2}, \\
&\frac{1}{a-2}\ \text{ if } \frac{1+\sqrt{21}}{2}\leqslant a\leqslant 3.
\endaligned
\right.
$$
\end{lemma}

\begin{proof}
Set $C=\mathbf{e}_3$. Then $\tau=2a-4$. Moreover, we have
$$
P(v)\sim_{\mathbb{R}} \left\{\aligned
&a\mathbf{h}-\mathbf{e}_1-\mathbf{e}_2-(2+v)\mathbf{e}_3  \ \text{ if } 0\leqslant v\leqslant a-2, \\
&(2a-v-2)\mathbf{h}-\mathbf{e}_1-\mathbf{e}_2+(v+2-2a)\mathbf{e}_3\ \text{ if } a-2\leqslant v\leqslant 2a-4,
\endaligned
\right.
$$
and
$$
N(v)= \left\{\aligned
&0\ \text{ if } 0\leqslant v\leqslant a-2, \\
&(v-a+2)\mathbf{l}_{3}\ \text{ if } a-2\leqslant v\leqslant 2a-4,
\endaligned
\right.
$$
which gives
$$
P(v)^2=\left\{\aligned
&\frac{2a^2-v^2-4v-8}{2} \ \text{ if } 0\leqslant v\leqslant a-2, \\
&\frac{(2a-v)(2a-v-4)}{2}\ \text{ if } a-2\leqslant v\leqslant 2a-4,
\endaligned
\right.
$$
and
$$
P(v)\cdot C=\left\{\aligned
&\frac{2+v}{2}\ \text{ if } 0\leqslant v\leqslant a-2, \\
&\frac{2a-v-2}{2}\ \text{ if } a-2\leqslant v\leqslant 2a-4.
\endaligned
\right.
$$
Integrating, we get $S_D(C)=a-2$ and
$$
S\big(W^C_{\bullet,\bullet};P\big)=\frac{2}{a^2-4}\int\limits_{a-2}^{2a-4}\mathrm{ord}_P\big(N(v)\vert_{C}\big)\big(P(v)\cdot C\big)dv+\frac{a^2+2a+4}{6(a+2)}.
$$
Thus, if $P\not\in\mathbf{l}_{3}$, then $S(W^C_{\bullet,\bullet};P)=\frac{a^2+2a+4}{6(a+2)}$.
Similarly, if $P\in\mathbf{l}_{3}$, then
$$
S\big(W^C_{\bullet,\bullet};P\big)=\frac{2}{a^2-4}\int\limits_{a-2}^{2a-4}(v+2-a)\big(P(v)\cdot C\big)dv+\frac{a^2+2a+4}{6(a+2)}=\frac{a^2+2a-2}{3(a+2)}.
$$
Now, using \eqref{equation:AZ-surface}, we obtain the required assertion.
\end{proof}

\begin{lemma}
\label{lemma:dP5-A1-e1-e2}
Let $P$ be a point in $\mathbf{e}_1\cup \mathbf{e}_2$ such that $P\not\in\mathbf{l}_1\cup\mathbf{l}_2$. Then
$$
\delta_P\big(S,D\big)\geqslant\frac{3(a+2)}{a^2+2a-2}.
$$
\end{lemma}

\begin{proof}
We may assume that $P\in\mathbf{e}_{1}$. Set $C=\mathbf{e}_{1}$. Then $\tau=2a-4$. Moreover, we have
$$
P(v)\sim_{\mathbb{R}} \left\{\aligned
&a\mathbf{h}-(1+v)\mathbf{e}_1-\mathbf{e}_2-2\mathbf{e}_3  \ \text{ if } 0\leqslant v\leqslant a-2, \\
&(4a-3v-6)\mathbf{h}+(2v+5-3a)\mathbf{e}_1+(1-a+v)\big(\mathbf{e}_2+2\mathbf{e}_3\big)\ \text{ if } a-2\leqslant v\leqslant 2a-4,
\endaligned
\right.
$$
and
$$
N(v)= \left\{\aligned
&0\ \text{ if } 0\leqslant v\leqslant a-2, \\
&(v-a+2)\big(2\mathbf{l}_{1}+\mathbf{l}_{4}\big)\ \text{ if } a-2\leqslant v\leqslant 2a-4,
\endaligned
\right.
$$
which gives
$$
P(v)^2=\left\{\aligned
&a^2-v^2-2v-4 \ \text{ if } 0\leqslant v\leqslant a-2, \\
&2(2a-v-4)(a-v-1)\ \text{ if } a-2\leqslant u\leqslant 2a-4,
\endaligned
\right.
$$
and
$$
P(v)\cdot C=\left\{\aligned
&1+v\ \text{ if } 0\leqslant v\leqslant a-2, \\
&3a-5-2v\ \text{ if } a-2\leqslant v\leqslant 2a-4.
\endaligned
\right.
$$
Integrating, we get $S_D(C)=\frac{(a+10)(a-2)}{3(a+2)}$ and
$$
S\big(W^C_{\bullet,\bullet};P\big)=\frac{2}{a^2-4}\int\limits_{a-2}^{2a-4}\mathrm{ord}_P\big(N(v)\vert_{C}\big)\big(P(v)\cdot C\big)dv+\frac{2a^2-5a+8}{3(a+2)}.
$$
Thus, if $P\not\in\mathbf{l}_{4}$, then $S(W^C_{\bullet,\bullet};P)=\frac{2a^2-5a+8}{3(a+2)}$.
Similarly, if $P\in\mathbf{l}_{4}$, then
$$
S\big(W^C_{\bullet,\bullet};P\big)=\frac{2}{a^2-4}\int\limits_{a-2}^{2a-4}(v+2-a)\big(P(v)\cdot C\big)dv+\frac{2a^2-5a+8}{3(a+2)}=\frac{a^2+2a-2}{3(a+2)}.
$$
Since $\frac{a^2+2a-2}{3(a+2)}\geqslant \frac{(a+10)(a-2)}{3(a+2)}$ for $a\in(2,3]$,
the required assertion follows from \eqref{equation:AZ-surface}.
\end{proof}

\begin{lemma}
\label{lemma:dP5-A1-L3}
Let $P$ be a point in $\mathbf{l}_3$ such that $P\not\in\mathbf{e}_3$. Then
$$
\delta_P\big(S,D\big)\geqslant\frac{3(a+2)}{a^2+2a-2}.
$$
\end{lemma}

\begin{proof}
Set $C=\mathbf{l}_{3}$. Then $\tau=a-1$. Moreover, we have
$$
P(v)\sim_{\mathbb{R}} \left\{\aligned
&(a-v)\mathbf{h}-\mathbf{e}_1-\mathbf{e}_2-(2-2v)\mathbf{e}_3\ \text{ if } 0\leqslant v\leqslant a-2, \\
&(2a-2v-2)\mathbf{h}+(1-a+v)\big(\mathbf{e}_1+\mathbf{e}_2\big)+(2v-2)\mathbf{e}_3\ \text{ if } a-2\leqslant v\leqslant 1, \\
&(2a-2v-2)\mathbf{h}+(1-a+v)\big(\mathbf{e}_1+\mathbf{e}_2\big)\ \text{ if } 1\leqslant v\leqslant a-1,
\endaligned
\right.
$$
and
$$
N(v)=\left\{\aligned
&0  \ \text{ if } 0\leqslant v\leqslant a-2, \\
&(v+2-a)\mathbf{l}_{4} \ \text{ if } a-2\leqslant v\leqslant 1, \\
&(v+2-a)\mathbf{l}_{4}+2(v-1)\mathbf{e}_3\ \text{ if } 1\leqslant v\leqslant a-1.
\endaligned
\right.
$$
This gives
$$
P(v)^2=\left\{\aligned
&a^2-2av-v^2+4v-4  \ \text{ if } 0\leqslant v\leqslant a-2, \\
&2(a-2)(a-2v) \ \text{ if } a-2\leqslant v\leqslant 1, \\
&2(a-v-1)^2\ \text{ if } 1\leqslant v\leqslant a-1,
\endaligned,
\right.
$$
and
$$
P(v)\cdot C=\left\{\aligned
&a+v-2  \ \text{ if } 0\leqslant v\leqslant a-2, \\
&2a-4 \ \text{ if } a-2\leqslant v\leqslant 1, \\
&2a-2v-2\ \text{ if } 1\leqslant v\leqslant a-1.
\endaligned,
\right.
$$
Now, integrating, we get $S_D(C)=\frac{a^2+2a-2}{3(a+2)}$, which gives $\delta_P(S,D)\leqslant\frac{3(a+2)}{a^2+2a-2}$.
Similarly, we get
$$
S\big(W^C_{\bullet,\bullet};P\big)=
\frac{2}{a^2-4}\int\limits_{a-2}^{a-1}\mathrm{ord}_P\big(N(v)\vert_{C}\big)\big(P(v)\cdot C\big)dv+\frac{(a-2)(14-a)}{3(a+2)}.
$$
Thus, if $P\not\in\mathbf{l}_{4}$, then $S(W^C_{\bullet,\bullet};P)=\frac{(a-2)(14-a)}{3(a+2)}$.
Similarly, if $P\in\mathbf{l}_{4}$, then
$$
S\big(W^C_{\bullet,\bullet};P\big)=\frac{2}{a^2-4}\int\limits_{a-2}^{a-1}(v+2-a)\big(P(v)\cdot C\big)dv+\frac{(a-2)(14-a)}{3(a+2)}=\frac{a^2+2a-2}{3(a+2)}=S_D(C),
$$
so the required assertion follows from \eqref{equation:AZ-surface}.
\end{proof}

\begin{lemma}
\label{lemma:dP5-A1-L4}
Let $P$ be a point in $\mathbf{l}_4$ such that $P\not\in\mathbf{e}_1\cup\mathbf{e}_2\cup\mathbf{l}_3$. Then
$$
\delta_P\big(S,D\big)\geqslant\frac{3(a+2)}{a^2+2a-2}.
$$
\end{lemma}

\begin{proof}
Set $C=\mathbf{l}_{4}$. Then $\tau=a-1$. Moreover, we have
$$
P(v)\sim_{\mathbb{R}} \left\{\aligned
&(a-v)\mathbf{h}-(1-v)\big(\mathbf{e}_1+\mathbf{e}_2\big)-2\mathbf{e}_3\ \text{ if } 0\leqslant v\leqslant a-2, \\
&(2a-2v-2)\mathbf{h}+(v-1)\big(\mathbf{e}_1+\mathbf{e}_2\big)+(2+2v-2a)\mathbf{e}_3\ \text{ if } a-2\leqslant v\leqslant 1, \\
&(2a-2v-2)\mathbf{h}+(2+2v-2a)\mathbf{e}_3\big)\ \text{ if } 1\leqslant v\leqslant a-1,
\endaligned
\right.
$$
and
$$
N(v)=\left\{\aligned
&0  \ \text{ if } 0\leqslant v\leqslant a-2, \\
&(v+2-a)\mathbf{l}_{3} \ \text{ if } a-2\leqslant v\leqslant 1, \\
&(v+2-a)\mathbf{l}_{3}+(v-1)\big(\mathbf{e}_1+\mathbf{e}_2\big)\ \text{ if } 1\leqslant v\leqslant a-1.
\endaligned
\right.
$$
This gives
$$
P(v)^2=\left\{\aligned
&a^2-2av-v^2+4v-4  \ \text{ if } 0\leqslant v\leqslant a-2, \\
&2(a-2)(a-2v) \ \text{ if } a-2\leqslant v\leqslant 1, \\
&2(a-v-1)^2\ \text{ if } 1\leqslant v\leqslant a-1,
\endaligned,
\right.
$$
and
$$
P(v)\cdot C=\left\{\aligned
&a+v-2  \ \text{ if } 0\leqslant v\leqslant a-2, \\
&2a-4 \ \text{ if } a-2\leqslant v\leqslant 1, \\
&2a-2v-2\ \text{ if } 1\leqslant v\leqslant a-1.
\endaligned,
\right.
$$
Now, integrating, we get $S_D(C)=\frac{a^2+2a-2}{3(a+2)}$, which gives $\delta_P(S,D)\leqslant\frac{3(a+2)}{a^2+2a-2}$.
Similarly, we get
$$
S\big(W^C_{\bullet,\bullet};P\big)=
\frac{2}{a^2-4}\int\limits_{a-2}^{a-1}\mathrm{ord}_P\big(N(v)\vert_{C}\big)\big(P(v)\cdot C\big)dv+\frac{(a-2)(14-a)}{3(a+2)}=\frac{(a-2)(14-a)}{3(a+2)},
$$
because $P\not\in\mathbf{e}_1\cup\mathbf{e}_2\cup\mathbf{l}_3$.
Since $\frac{(a-2)(14-a)}{3(a+2)}\leqslant S_D(C)$, we have $\delta_P(S,D)=\frac{3(a+2)}{a^2+2a-2}$ by \eqref{equation:AZ-surface}.
\end{proof}

\begin{lemma}
\label{lemma:dP5-A1-general-point}
Suppose that $P\not\in\mathbf{e}_1\cup\mathbf{e}_2\cup\mathbf{e}_3\cup\mathbf{l}_{1}\cup\mathbf{l}_{2}\cup\mathbf{l}_{3}\cup\mathbf{l}_{4}$.
Then
$$
\delta_P\big(S,D\big)\geqslant
\left\{\aligned
&\frac{2(a+2)}{a^2-2a+4} \ \text{ if } 2<a\leqslant 5-\sqrt{5}, \\
&\frac{2(2a+4)}{a^2+6a-12}\ \text{ if }  5-\sqrt{5}\leqslant u\leqslant 3.
\endaligned
\right.
$$
\end{lemma}

\begin{proof}
Let $\widetilde{\mathbf{e}}_1$, $\widetilde{\mathbf{e}}_2$, $\widetilde{\mathbf{e}}_3$, $\widetilde{\mathbf{l}}_{1}$, $\widetilde{\mathbf{l}}_{2}$, $\widetilde{\mathbf{l}}_{3}$, $\widetilde{\mathbf{l}}_{4}$
be the strict transforms on $\widetilde{S}$ of the curves $\mathbf{e}_1$, $\mathbf{e}_2$, $\mathbf{e}_3$, $\mathbf{l}_{1}$, $\mathbf{l}_{2}$, $\mathbf{l}_{3}$, $\mathbf{l}_{4}$, respectively.
Set $\widetilde{\mathbf{h}}=f^*(\mathbf{h})$.
Let $\mathbf{c}_0$, $\mathbf{c}_1$, $\mathbf{c}_2$, $\mathbf{c}_3$
be the curves in $|2\widetilde{\mathbf{h}}-E-\widetilde{\mathbf{e}}_1-\widetilde{\mathbf{e}}_2-2\widetilde{\mathbf{e}}_3|$,
$|\widetilde{\mathbf{h}}-E-\widetilde{\mathbf{e}}_1|$, $|\widetilde{\mathbf{h}}-E-\widetilde{\mathbf{e}}_2|$,
$|\widetilde{\mathbf{h}}-E-\widetilde{\mathbf{e}}_3|$, respectively.
Then
\begin{center}
$\mathbf{c}_0$, $\mathbf{c}_1$, $\mathbf{c}_2$, $\mathbf{c}_3$, $\widetilde{\mathbf{e}}_1$, $\widetilde{\mathbf{e}}_2$, $\widetilde{\mathbf{e}}_3$, $\widetilde{\mathbf{l}}_{1}$, $\widetilde{\mathbf{l}}_{2}$, $\widetilde{\mathbf{l}}_{3}$, $\widetilde{\mathbf{l}}_{4}$, $E$
\end{center}
are all curves in $\widetilde{S}$ that have negative self-intersections \cite{CorayTsfasman}.

We compute $\widetilde{\tau}=\frac{3a-4}{2}$. Similarly, we see that
$$
\widetilde{P}(v)\sim_{\mathbb{R}} \left\{\aligned
&a\widetilde{\mathbf{h}}-\widetilde{\mathbf{e}}_1-\widetilde{\mathbf{e}}_2-2\widetilde{\mathbf{e}}_3-vE \ \text{ if } 0\leqslant v\leqslant 2a-4, \\
&(5a-2v-8)\widetilde{\mathbf{h}}+(3-2a+v)\big(\widetilde{\mathbf{e}}_1+\widetilde{\mathbf{e}}_2+2\widetilde{\mathbf{e}}_3\big)-(2a-4)E\ \text{ if } 2a-4\leqslant v\leqslant a-1, \\
&(3a-2v-4)\big(3\widetilde{\mathbf{h}}-\widetilde{\mathbf{e}}_1-\widetilde{\mathbf{e}}_2-2\widetilde{\mathbf{e}}_3-2E\big) \ \text{ if } a-1\leqslant v\leqslant \frac{3a-4}{2},
\endaligned
\right.
$$
and
$$
\widetilde{N}(v)=\left\{\aligned
&0\ \text{ if } 0\leqslant v\leqslant 2a-4, \\
&(v+4-2a)\mathbf{c}_0 \ \text{ if } 2a-4\leqslant v\leqslant a-1, \\
&(v+4-2a)\mathbf{c}_0+(v+1-a)\big(\mathbf{c}_1+\mathbf{c}_2+2\mathbf{c}_3\big) \ \text{ if } a-1\leqslant v\leqslant \frac{3a-4}{2}.
\endaligned
\right.
$$
This gives
$$
\widetilde{P}(v)^2=\left\{\aligned
&a^2-v^2-4\ \text{ if } 0\leqslant v\leqslant 2a-4, \\
&(a-2)(5a-4v-6) \ \text{ if } 2a-4\leqslant v\leqslant a-1, \\
&(3a-2v-4)^2 \ \text{ if } a-1\leqslant v\leqslant \frac{3a-4}{2},
\endaligned
\right.
$$
and
$$
\widetilde{P}(v)\cdot E=\left\{\aligned
&v \ \text{ if } 0\leqslant v\leqslant 2a-4, \\
&2a-4 \ \text{ if } 2a-4\leqslant v\leqslant a-1, \\
&6a-4v-8\ \text{ if } a-1\leqslant v\leqslant \frac{3a-4}{2}.
\endaligned
\right.
$$
Now, integrating, we get $S_D(E)=\frac{a^2+6a-12}{2(a+2)}$.

Let $O$ be a point in $E$. Then
$$
S\big(W^E_{\bullet,\bullet};O\big)=
\frac{2}{a^2-4}\int\limits_{2a-4}^{\frac{3a-4}{2}}\mathrm{ord}_O\big(\widetilde{N}(v)\vert_{E}\big)\big(\widetilde{P}(v)\cdot E\big)dv+\frac{2(8-a)(a-2)}{3(a+2)}.
$$
Thus, if $O\not\in\mathbf{c}_0\cup\mathbf{c}_1\cup\mathbf{c}_2\cup\mathbf{c}_3$, then $S(W^E_{\bullet,\bullet};O)=\frac{2(8-a)(a-2)}{3(a+2)}$.
Similarly, if $O\in\mathbf{c}_0$, then
$$
S\big(W^E_{\bullet,\bullet};O\big)=
\frac{2}{a^2-4}\int\limits_{2a-4}^{\frac{3a-4}{2}}(v+4-2a)\big(\widetilde{P}(v)\cdot E\big)dv+\frac{2(8-a)(a-2)}{3(a+2)}=\frac{a^2-2a+4}{2(a+2)}.
$$
Likewise, if $O\in\mathbf{c}_1\cup\mathbf{c}_2$, then
$$
S\big(W^E_{\bullet,\bullet};O\big)=
\frac{2}{a^2-4}\int\limits_{a-1}^{\frac{3a-4}{2}}(v + 1 - a)\big(\widetilde{P}(v)\cdot E\big)dv+\frac{2(8-a)(a-2)}{3(a+2)}=\frac{(a-2)(10-a)}{2(a+2)}.
$$
Finally, if $O\in\mathbf{c}_3$, then
$$
S\big(W^E_{\bullet,\bullet};O\big)=
\frac{2}{a^2-4}\int\limits_{a-1}^{\frac{3a-4}{2}}2(v+1-a)\big(\widetilde{P}(v)\cdot E\big)dv+\frac{2(8-a)(a-2)}{3(a+2)}=\frac{(a-2)(14-a)}{3(a+2)}.
$$
Therefore, using \eqref{equation:AZ-surface-blow-up}, we get
$$
\delta_P(S,D)\geqslant\min\Bigg\{\frac{4(a+2)}{a^2+6a-12},\frac{2(a+2)}{a^2-2a+4},\frac{2(a+2)}{(a-2)(10-a)},\frac{3(a+2)}{(a-2)(14-a)}\Bigg\},
$$
which implies the required assertion.
\end{proof}

Combining Lemmas~\ref{lemma:dP5-A1-e3}, \ref{lemma:dP5-A1-e1-e2}, \ref{lemma:dP5-A1-L3}, \ref{lemma:dP5-A1-L4}, \ref{lemma:dP5-A1-general-point}, we obtain

\begin{corollary}
\label{corollary:dP5-A1}
Let $P$ be a point in $S$ such that $P\not\in\mathbf{l}_1\cup\mathbf{l}_2$. If $P\in\mathbf{e}_1\cup\mathbf{e}_2\cup\mathbf{e}_3$, then
$$
\delta_P\big(S,D\big)\geqslant
\left\{\aligned
&\frac{3(a+2)}{a^2+2a-2} \ \text{ if } 2<a\leqslant \frac{1+\sqrt{21}}{2}, \\
&\frac{1}{a-2}\ \text{ if } \frac{1+\sqrt{21}}{2}\leqslant u\leqslant 3.
\endaligned
\right.
$$
If $P\not\in\mathbf{e}_1\cup\mathbf{e}_2\cup\mathbf{e}_3$, then $\delta_P(S,D)\geqslant\frac{3(a+2)}{a^2+2a-2}$ for every $a\in(2,3]$.
\end{corollary}

\subsection{Quintic del Pezzo surface with singular point of type $\mathbb{A}_2$}
\label{subsection:dP5-A2}

Let $S$ be a del Pezzo surface such that $K_S^2=5$, and $S$ has one singular point,
which is a singular point of type $\mathbb{A}_1$. Set~$O=\mathrm{Sing}(S)$.
Then it follows from \cite{CorayTsfasman} that there exists a birational morphism $\pi\colon S\to\mathbb{P}^2$ that contracts
smooth irreducible rational curves $\mathbf{e}_1$ and $\mathbf{e}_2$ such that $\mathbf{e}_1\cap\mathbf{e}_2=\varnothing$,
$O\not\in\mathbf{e}_1$,  $O\in\mathbf{e}_2$, $\mathbf{e}_1^2=-1$, and $\mathbf{e}_3^2=-\frac{1}{3}$.
Set~$\mathbf{h}=\pi^*(\mathcal{O}_{\mathbb{P}^2}(1))$.
Let $\mathbf{l}_1$ and $\mathbf{l}_2$ be irreducible curves in $|\mathbf{h}-2\mathbf{e}_2|$ and $|\mathbf{h}-\mathbf{e}_1-\mathbf{e}_2|$, respectively.
Then $\mathbf{l}_1$ and $\mathbf{l}_2$ are smooth and rational, $O=\mathbf{l}_1\cap\mathbf{l}_2\cap\mathbf{e}_2$,
and $\mathbf{e}_1$, $\mathbf{e}_2$, $\mathbf{l}_1$, $\mathbf{l}_2$ are all curves in $S$ that have negative self-intersections.
They intersections are given in the following table:

\begin{center}
\renewcommand\arraystretch{1.5}
\begin{tabular}{|c||c|c|c|c|}
\hline
$\bullet$      & $\mathbf{e}_1$ & $\mathbf{e}_2$ & $\mathbf{l}_1$ & $\mathbf{l}_2$\\
\hline\hline
$\mathbf{e}_1$ &    $-1$        &     $0$        &      $0$       &    $1$      \\
\hline
$\mathbf{e}_2$ &     $0$        &  $-\frac{1}{3}$&  $\frac{2}{3}$ &    $\frac{1}{3}$ \\
\hline
$\mathbf{l}_1$ &     $0$        &  $\frac{2}{3}$ &   $-\frac{1}{3}$&  $\frac{1}{3}$   \\
\hline
$\mathbf{l}_2$ &     $1$        & $\frac{1}{3}$  &   $\frac{1}{3}$& -$\frac{1}{3}$  \\
\hline
\end{tabular}
\end{center}

Set $D=a\mathbf{h}-\mathbf{e}_1-3\mathbf{e}_2$ for $a\in(2,3]$. Then $D$ is ample and $D^2=a^2-4$.

\begin{lemma}
\label{lemma:dP5-A2-e2}
Let $P$ be a point in $\mathbf{e}_2$ such that $P\ne O$. Then
$$
\delta_P\big(S,D\big)\geqslant
\left\{\aligned
&\frac{6(a+2)}{a^2+2a+4} \ \text{ if } 2<a\leqslant \frac{1+\sqrt{17}}{2}, \\
&\frac{6(a+2)}{(7a+10)(a-2)}\ \text{ if } \frac{1+\sqrt{17}}{2}\leqslant a\leqslant 3.
\endaligned
\right.
$$
\end{lemma}

\begin{proof}
Set $C=\mathbf{e}_2$. Then $\tau=2a-4$. Moreover, we have
$$
P(v)\sim_{\mathbb{R}} \left\{\aligned
&a\mathbf{h}-\mathbf{e}_1-(3+v)\mathbf{e}_2 \ \text{ if } 0\leqslant v\leqslant \frac{3a-6}{2}, \\
&(4a-6-2v)\mathbf{h}-\mathbf{e}_1+(3v+9-6a)\mathbf{e}_2\ \text{ if } \frac{3a-6}{2}\leqslant v\leqslant 2a-4,
\endaligned
\right.
$$
and
$$
N(v)= \left\{\aligned
&0\ \text{ if } 0\leqslant v\leqslant \frac{3a-6}{2}, \\
&(6-3a+2v)\mathbf{l}_{1}\ \text{ if } \frac{3a-6}{2}\leqslant v\leqslant 2a-4,
\endaligned
\right.
$$
which gives
$$
P(v)^2=\left\{\aligned
&\frac{3a^2-v^2-6v-12}{3} \ \text{ if } 0\leqslant v\leqslant \frac{3a-6}{2}, \\
&(2a-v-2)(2a-v-4)\ \text{ if } \frac{3a-6}{2}\leqslant v\leqslant 2a-4,
\endaligned
\right.
$$
and
$$
P(v)\cdot C=\left\{\aligned
&\frac{3+v}{3}\ \text{ if } 0\leqslant v\leqslant \frac{3a-6}{2}, \\
&2a-3-v\ \text{ if } \frac{3a-6}{2}\leqslant v\leqslant 2a-4.
\endaligned
\right.
$$
Integrating, we get $S_D(C)=\frac{(7a+10)(a-2)}{6(a+2)}$.
Similarly, we compute $S(W^C_{\bullet,\bullet};P)=\frac{a^2+2a+4}{6(a+2)}$, since $P\not\in\mathbf{l}_{1}$.
Now, using \eqref{equation:AZ-surface}, we obtain the required inequality.
\end{proof}

\begin{lemma}
\label{lemma:dP5-A2-e1-e2}
Let $P$ be a point in $\mathbf{e}_1$ such that $P\not\in\mathbf{l}_2$. Then
$$
\delta_P\big(S,D\big)\geqslant
\left\{\aligned
&\frac{3(a+2)}{2a^2-5a+8} \ \text{ if } 2<a\leqslant \frac{13+\sqrt{57}}{2}, \\
&\frac{3(a+2)}{(a+10)(a-2)}\ \text{ if } \frac{13+\sqrt{57}}{2}\leqslant a\leqslant 3.
\endaligned
\right.
$$
\end{lemma}

\begin{proof}
Set $C=\mathbf{e}_{1}$. Then $\tau=2a-4$. Moreover, we have
$$
P(v)\sim_{\mathbb{R}} \left\{\aligned
&a\mathbf{h}-(1+v)\mathbf{e}_1-3\mathbf{e}_2  \ \text{ if } 0\leqslant v\leqslant a-2, \\
&(4a-3v-6)\mathbf{h}+(2v+5-3a)\mathbf{e}_1+3(1-a+v)\mathbf{e}_2\ \text{ if } a-2\leqslant v\leqslant 2a-4,
\endaligned
\right.
$$
and
$$
N(v)= \left\{\aligned
&0\ \text{ if } 0\leqslant v\leqslant a-2, \\
&3(v-a+2)\mathbf{l}_{2}\ \text{ if } a-2\leqslant v\leqslant 2a-4,
\endaligned
\right.
$$
which gives
$$
P(v)^2=\left\{\aligned
&a^2-v^2-2v-4 \ \text{ if } 0\leqslant v\leqslant a-2, \\
&2(2a-v-4)(a-v-1)\ \text{ if } a-2\leqslant v\leqslant 2a-4,
\endaligned
\right.
$$
and
$$
P(v)\cdot C=\left\{\aligned
&1+v\ \text{ if } 0\leqslant v\leqslant a-2, \\
&3a-5-2v\ \text{ if } a-2\leqslant v\leqslant 2a-4.
\endaligned
\right.
$$
This gives $S_D(C)=\frac{(a+10)(a-2)}{3(a+2)}$ and
$S(W^C_{\bullet,\bullet};P)=\frac{2a^2-5a+8}{3(a+2)}$, so \eqref{equation:AZ-surface} implies the required assertion.
\end{proof}

\begin{lemma}
\label{lemma:dP5-A2-general-point}
Suppose that $P\not\in\mathbf{e}_1\cup\mathbf{e}_2\cup\mathbf{l}_{1}\cup\mathbf{l}_{2}$.
Then
$$
\delta_P\big(S,D\big)\geqslant
\left\{\aligned
&\frac{2(a+2)}{a^2-2a+4} \ \text{ if } 2<a\leqslant 5-\sqrt{5}, \\
&\frac{2(2a+4)}{a^2+6a-12}\ \text{ if }  5-\sqrt{5}\leqslant a\leqslant \frac{19-\sqrt{21}}{5},\\
&\frac{6(a+2)}{(a-2)(26-a)}\ \text{ if }  \frac{19-\sqrt{21}}{5}\leqslant a\leqslant 3.
\endaligned
\right.
$$
\end{lemma}

\begin{proof}
Let $\widetilde{\mathbf{e}}_1$, $\widetilde{\mathbf{e}}_2$, $\widetilde{\mathbf{l}}_{1}$ and $\widetilde{\mathbf{l}}_{2}$,
be the strict transforms on $\widetilde{S}$ of the curves $\mathbf{e}_1$, $\mathbf{e}_2$, $\mathbf{l}_{1}$ and $\mathbf{l}_{2}$, respectively.
Set $\widetilde{\mathbf{h}}=f^*(\mathbf{h})$.
Let $\mathbf{c}_0$, $\mathbf{c}_1$, $\mathbf{c}_2$
be the curves in $|2\widetilde{\mathbf{h}}-E-\widetilde{\mathbf{e}}_1-3\widetilde{\mathbf{e}}_2|$,
$|\widetilde{\mathbf{h}}-E-\widetilde{\mathbf{e}}_1|$, $|\widetilde{\mathbf{h}}-E-\widetilde{\mathbf{e}}_2|$,~respectively.
Then
$\mathbf{c}_0$, $\mathbf{c}_1$, $\mathbf{c}_2$, $\widetilde{\mathbf{e}}_1$, $\widetilde{\mathbf{e}}_2$, $\widetilde{\mathbf{l}}_{1}$, $\widetilde{\mathbf{l}}_{2}$, $E$
are all curves in $\widetilde{S}$ that have negative self-intersections \cite{CorayTsfasman}.

We compute $\widetilde{\tau}=\frac{3a-4}{2}$. Similarly, we see that
$$
\widetilde{P}(v)\sim_{\mathbb{R}} \left\{\aligned
&a\widetilde{\mathbf{h}}-\widetilde{\mathbf{e}}_1-3\widetilde{\mathbf{e}}_2-vE \ \text{ if } 0\leqslant v\leqslant 2a-4, \\
&(5a-2v-8)\widetilde{\mathbf{h}}+(3-2a+v)\big(\widetilde{\mathbf{e}}_1+3\widetilde{\mathbf{e}}_2\big)+(4-2a)E\ \text{ if } 2a-4\leqslant v\leqslant a-1, \\
&(3a-2v-4)\big(3\widetilde{\mathbf{h}}-\widetilde{\mathbf{e}}_1-3\widetilde{\mathbf{e}}_2-2E\big) \ \text{ if } a-1\leqslant v\leqslant \frac{3a-4}{2},
\endaligned
\right.
$$
and
$$
\widetilde{N}(v)=\left\{\aligned
&0\ \text{ if } 0\leqslant v\leqslant 2a-4, \\
&(v+4-2a)\mathbf{c}_0 \ \text{ if } 2a-4\leqslant v\leqslant a-1, \\
&(v+4-2a)\mathbf{c}_0+(v+1-a)\big(\mathbf{c}_1+3\mathbf{c}_2\big) \ \text{ if } a-1\leqslant v\leqslant \frac{3a-4}{2}.
\endaligned
\right.
$$
This gives
$$
\widetilde{P}(v)^2=\left\{\aligned
&a^2-v^2-4\ \text{ if } 0\leqslant v\leqslant 2a-4, \\
&(a-2)(5a-4v-6) \ \text{ if } 2a-4\leqslant v\leqslant a-1, \\
&(3a-2v-4)^2 \ \text{ if } a-1\leqslant v\leqslant \frac{3a-4}{2},
\endaligned
\right.
$$
and
$$
\widetilde{P}(v)\cdot E=\left\{\aligned
&v \ \text{ if } 0\leqslant v\leqslant 2a-4, \\
&2a-4 \ \text{ if } 2a-4\leqslant v\leqslant a-1, \\
&6a-4v-8\ \text{ if } a-1\leqslant v\leqslant \frac{3a-4}{2}.
\endaligned
\right.
$$
Now, integrating, we get $S_D(E)=\frac{a^2+6a-12}{2(a+2)}$.

Let $O$ be a point in $E$. Then
$$
S\big(W^E_{\bullet,\bullet};O\big)=
\frac{2}{a^2-4}\int\limits_{2a-4}^{\frac{3a-4}{2}}\mathrm{ord}_O\big(\widetilde{N}(v)\vert_{E}\big)\big(\widetilde{P}(v)\cdot E\big)dv+\frac{2(8-a)(a-2)}{3(a+2)}.
$$
Thus, if $O\not\in\mathbf{c}_0\cup\mathbf{c}_1\cup\mathbf{c}_2$, then $S(W^E_{\bullet,\bullet};O)=\frac{2(8-a)(a-2)}{3(a+2)}$.
Similarly, if $O\in\mathbf{c}_0$, then
$$
S\big(W^E_{\bullet,\bullet};O\big)=
\frac{2}{a^2-4}\int\limits_{2a-4}^{\frac{3a-4}{2}}(v+4-2a)\big(\widetilde{P}(v)\cdot E\big)dv+\frac{2(8-a)(a-2)}{3(a+2)}=\frac{a^2-2a+4}{2(a+2)}.
$$
Likewise, if $O\in\mathbf{c}_1$, then
$$
S\big(W^E_{\bullet,\bullet};O\big)=
\frac{2}{a^2-4}\int\limits_{a-1}^{\frac{3a-4}{2}}(v + 1 - a)\big(\widetilde{P}(v)\cdot E\big)dv+\frac{2(8-a)(a-2)}{3(a+2)}=\frac{(a-2)(10-a)}{2(a+2)}.
$$
Finally, if $O\in\mathbf{c}_2$, then
$$
S\big(W^E_{\bullet,\bullet};O\big)=
\frac{2}{a^2-4}\int\limits_{a-1}^{\frac{3a-4}{2}}3(v+1-a)\big(\widetilde{P}(v)\cdot E\big)dv+\frac{2(8-a)(a-2)}{3(a+2)}=\frac{(a-2)(26-a)}{6(a+2)}.
$$
Therefore, using \eqref{equation:AZ-surface-blow-up}, we get
$$
\delta_P(S,D)\geqslant\min\Bigg\{\frac{4(a+2)}{a^2+6a-12},\frac{2(a+2)}{a^2-2a+4},\frac{2(a+2)}{(a-2)(10-a)},\frac{6(a+2)}{(a-2)(26-a)}\Bigg\},
$$
which implies the required assertion.
\end{proof}

Combining Lemmas~\ref{lemma:dP5-A2-e2}, \ref{lemma:dP5-A2-e1-e2}, \ref{lemma:dP5-A2-general-point}, we obtain

\begin{corollary}
\label{corollary:dP5-A2}
Let $P$ be a point in $S$ such that $P\not\in\mathbf{l}_1\cup\mathbf{l}_2$. If $P\in\mathbf{e}_1\cup\mathbf{e}_2$, then
$$
\delta_P\big(S,D\big)\geqslant
\left\{\aligned
&\frac{6(a+2)}{a^2+2a+4} \ \text{ if } 2<a\leqslant \frac{1+\sqrt{17}}{2}, \\
&\frac{6(a+2)}{(7a+10)(a-2)}\ \text{ if } \frac{1+\sqrt{17}}{2}\leqslant a\leqslant 3.
\endaligned
\right.
$$
If $P\not\in\mathbf{e}_1\cup\mathbf{e}_2$, then
$$
\delta_P\big(S,D\big)\geqslant
\left\{\aligned
&\frac{2(a+2)}{a^2-2a+4} \ \text{ if } 2<a\leqslant 5-\sqrt{5}, \\
&\frac{2(2a+4)}{a^2+6a-12}\ \text{ if }  5-\sqrt{5}\leqslant a\leqslant \frac{19-\sqrt{21}}{5},\\
&\frac{6(a+2)}{(a-2)(26-a)}\ \text{ if }  \frac{19-\sqrt{21}}{5}\leqslant a\leqslant 3.
\endaligned
\right.
$$
\end{corollary}

\subsection{Smooth sextic del Pezzo surface}
\label{subsection:dP6-smooth}

Let $\ell_1$, $\ell_2$, $\ell_3$, $\ell_4$ be four distinct rulings of $\mathbb{P}^1\times\mathbb{P}^1$
such that $\ell_1\cap\ell_3=\varnothing$, $\ell_2\cap\ell_4=\varnothing$,
and each intersection $\ell_1\cap\ell_2$, $\ell_2\cap\ell_3$, $\ell_3\cap\ell_4$, $\ell_4\cap\ell_1$ consists of one point,
let $\pi\colon S\to\mathbb{P}^1\times\mathbb{P}^1$ be the blow up of the points $\ell_1\cap\ell_2$ and $\ell_3\cap\ell_4$,
let $\mathbf{e}_1$ and $\mathbf{e}_2$ be the $\pi$-exceptional curves such that $\pi(\mathbf{e}_1)=\ell_1\cap\ell_2$
and $\pi(\mathbf{e}_2)=\ell_3\cap\ell_4$, let $\mathbf{l}_1$, $\mathbf{l}_2$, $\mathbf{l}_3$, $\mathbf{l}_4$
be the strict transforms on $S$ of the curves $\ell_1$, $\ell_2$, $\ell_3$, $\ell_4$, respectively.
Then $S$ is a del Pezzo surface of degree $6$, and
$\mathbf{e}_1$, $\mathbf{e}_2$, $\mathbf{l}_1$, $\mathbf{l}_2$, $\mathbf{l}_3$, $\mathbf{l}_4$
are all $(-1)$-curves in $S$. Set~$\mathbf{h}_1=\pi^*(\ell_1)$ and $\mathbf{h}_2=\pi^*(\ell_2)$. Set
$$
D=a\big(\mathbf{h}_1+\mathbf{h}_2)-\mathbf{e}_1-\mathbf{e}_2
$$
for $a\in(1,2]$. Then $D$ is ample and $D^2=2a^2-2$.

\begin{lemma}
\label{lemma:dP6-e1-e2}
Let $P$ be a point in $\mathbf{e}_1\cup\mathbf{e}_2$. If $P\in\mathbf{l}_1\cup\mathbf{l}_2\cup\mathbf{l}_3\cup\mathbf{l}_4$, then
$\delta_P(S,D)\geqslant\frac{2}{a}$ for $a\in(1,2]$. Similarly, if $P\not\in\mathbf{l}_1\cup\mathbf{l}_2\cup\mathbf{l}_3\cup\mathbf{l}_4$, then
$$
\delta_P\big(S,D\big)\geqslant
\left\{\aligned
&\frac{3(a+1)}{a^2+a+1} \ \text{ if } 1<a\leqslant \frac{1+\sqrt{33}}{4}, \\
&\frac{1}{a-1}\ \text{ if } \frac{1+\sqrt{33}}{4}\leqslant a\leqslant 2.
\endaligned
\right.
$$
\end{lemma}

\begin{proof}
We may assume that $P\in\mathbf{e}_1$. Set $C=\mathbf{e}_1$. Then $\tau=2a-2$. Moreover, we have
$$
P(v)\sim_{\mathbb{R}} \left\{\aligned
&a\big(\mathbf{h}_1+\mathbf{h}_2)-(1+v)\mathbf{e}_1-\mathbf{e}_2 \ \text{ if } 0\leqslant v\leqslant a-1, \\
&(2a-v-1)\big(\mathbf{h}_1+\mathbf{h}_2)+(v+1-2a)\mathbf{e}_1-\mathbf{e}_2\ \text{ if } a-1\leqslant v\leqslant 2a-2,
\endaligned
\right.
$$
and
$$
N(v)=\left\{\aligned
&0\ \text{ if } 0\leqslant v\leqslant a-1, \\
&(v+1-a)\big(\mathbf{l}_{1}+\mathbf{l}_{2}\big)\ \text{ if } a-1\leqslant v\leqslant 2a-2,
\endaligned
\right.
$$
which gives
$$
P(v)^2=\left\{\aligned
&2a^2-v^2-2v-2 \ \text{ if } 0\leqslant v\leqslant a-1, \\
&(2a-v)(2a-v-2)\ \text{ if } a-1\leqslant v\leqslant 2a-2,
\endaligned
\right.
$$
and
$$
P(v)\cdot C=\left\{\aligned
&1+v\ \text{ if } 0\leqslant v\leqslant a-1, \\
&2a-v-1\ \text{ if } a-1\leqslant v\leqslant 2a-2.
\endaligned
\right.
$$
Integrating, we get $S_D(C)=a-1$. Similarly, we get
$$
S\big(W^C_{\bullet,\bullet};P\big)=
\frac{1}{a^2-1}\int\limits_{a-1}^{2a-2}\mathrm{ord}_P\big(N(v)\vert_{C}\big)\big(P(v)\cdot C\big)dv+\frac{a^2+a+1}{3(a+1)}.
$$
Thus, if $P\not\in\mathbf{l}_{1}\cup\mathbf{l}_{2}$, then $S(W^C_{\bullet,\bullet};P)=\frac{a^2+a+1}{3(a+1)}$, so that \eqref{equation:AZ-surface} gives
$$
\delta_P(S,D)\geqslant\min\Bigg\{\frac{1}{a-1},\frac{3(a+1)}{a^2+a+1}\Bigg\}=\left\{\aligned
&\frac{3(a+1)}{a^2+a+1} \ \text{ if } 1<a\leqslant \frac{1+\sqrt{33}}{4}, \\
&\frac{1}{a-1}\ \text{ if } \frac{1+\sqrt{33}}{4}\leqslant a\leqslant 2.
\endaligned
\right.
$$
Similarly, if $P\in\mathbf{l}_{1}\cup\mathbf{l}_{2}$, then
$$
S\big(W^C_{\bullet,\bullet};P\big)=
\frac{1}{a^2-1}\int\limits_{a-1}^{2a-2}(v+1-a)\big(P(v)\cdot C\big)dv+\frac{a^2+a+1}{3(a+1)}=\frac{a}{2},
$$
so that $S(W^C_{\bullet,\bullet};P)\leqslant\frac{a}{2}$ and $S_D(C)\leqslant\frac{a}{2}$,
which gives $\delta_P(S,D)\geqslant\frac{2}{a}$ by \eqref{equation:AZ-surface}.
\end{proof}

\begin{lemma}
\label{lemma:dP6-lines}
Let $P$ be a point in $\mathbf{l}_1\cup\mathbf{l}_2\cup\mathbf{l}_3\cup\mathbf{l}_4$.
Then $\delta_P(S,D)\geqslant\frac{2}{a}$ for every $a\in(1,2]$.
\end{lemma}

\begin{proof}
We may assume that $P\in\mathbf{l}_1$.
Moreover, by Lemma~\ref{lemma:dP6-e1-e2}, we may assume that $P\not\in\mathbf{e}_1\cup\mathbf{e}_2$.
Set $C=\mathbf{l}_1$. Then $\tau=a$. Moreover, we have
$$
P(v)\sim_{\mathbb{R}} \left\{\aligned
&(a-v)\mathbf{h}_1+a\mathbf{h}_2+(v-1)\mathbf{e}_1-\mathbf{e}_2 \ \text{ if } 0\leqslant v\leqslant a-1, \\
&(a-v)\mathbf{h}_1+(2a-v-1)\mathbf{h}_2+(v-1)\mathbf{e}_1+(v-a)\mathbf{e}_2\ \text{ if } a-1\leqslant v\leqslant 1,\\
&(a-v)\mathbf{h}_1+(2a-v-1)\mathbf{h}_2+(v-a)\mathbf{e}_2\ \text{ if } 1\leqslant v\leqslant a,
\endaligned
\right.
$$
and
$$
N(v)=\left\{\aligned
&\ \text{ if } 0\leqslant v\leqslant a-1, \\
&(v+1-a)\mathbf{l}_4\ \text{ if } a-1\leqslant v\leqslant 1,\\
&(v+1-a)\mathbf{l}_4+(v-1)\mathbf{e}_1\ \text{ if } 1\leqslant v\leqslant a,
\endaligned
\right.
$$
which gives
$$
P(v)^2=\left\{\aligned
&2a^2-2av-v^2+2v-2 \ \text{ if } 0\leqslant v\leqslant a-1, \\
&(a-1)(3a-4v+1)\ \text{ if } a-1\leqslant v\leqslant 1,\\
&(3a-v-2)(a-v)\ \text{ if } 1\leqslant v\leqslant a,
\endaligned
\right.
$$
and
$$
P(v)\cdot C=\left\{\aligned
&a+v-1\ \text{ if } 0\leqslant v\leqslant a-1, \\
&2a-2\ \text{ if } a-1\leqslant v\leqslant 1,\\
&2a-v-1\ \text{ if } 1\leqslant v\leqslant a.
\endaligned
\right.
$$
Integrating, we get $S_D(C)=\frac{a}{2}$. Similarly, we get
$$
S\big(W^C_{\bullet,\bullet};P\big)=
\frac{1}{a^2-1}\int\limits_{a-1}^{a}\mathrm{ord}_P\big(N(v)\vert_{C}\big)\big(P(v)\cdot C\big)dv+\frac{(a+5)(a-1)}{3(a+1)}.
$$
Thus, if $P\not\in\mathbf{l}_{4}\cup\mathbf{e}_{1}$, then $S(W^C_{\bullet,\bullet};P)=\frac{(a+5)(a-1)}{3(a+1)}$.
Similarly, if $P\in\mathbf{l}_{4}$, then
$$
S\big(W^C_{\bullet,\bullet};P\big)=
\frac{1}{a^2-1}\int\limits_{a-1}^{a}(v+1-a)\big(P(v)\cdot C\big)dv+\frac{(a+5)(a-1)}{3(a+1)}=\frac{a}{2}.
$$
Hence, we see that $S(W^C_{\bullet,\bullet};P)\leqslant S_D(C)\geqslant\frac{a}{2}$,
so that $\delta_P(S,D)=\frac{2}{a}$ by \eqref{equation:AZ-surface}.
\end{proof}

\begin{lemma}
\label{lemma:dP6-general-point}
Suppose that $P\not\in\mathbf{e}_1\cup\mathbf{e}_2\cup\mathbf{l}_{1}\cup\mathbf{l}_{2}\cup\mathbf{l}_{3}\cup\mathbf{l}_{4}$.
Then
$$
\delta_P\big(S,D\big)\geqslant
\left\{\aligned
&\frac{3(a+1)}{a^2+a+1} \ \text{ if } 1<a\leqslant \frac{\sqrt{21}-1}{2}, \\
&\frac{2(a+1)}{a^2+a-1}\ \text{ if } \frac{\sqrt{21}-1}{2}\leqslant a\leqslant 2.
\endaligned
\right.
$$
\end{lemma}

\begin{proof}
Recall that $\widetilde{S}$ is a smooth del Pezzo surface of degree $5$.
Let $\widetilde{\mathbf{e}}_1$, $\widetilde{\mathbf{e}}_2$, $\widetilde{\mathbf{l}}_{1}$, $\widetilde{\mathbf{l}}_{2}$, $\widetilde{\mathbf{l}}_{3}$,
$\widetilde{\mathbf{l}}_{4}$ be the strict transforms on $\widetilde{S}$ of the $(-1)$-curves $\mathbf{e}_1$, $\mathbf{e}_2$, $\mathbf{l}_{1}$, $\mathbf{l}_{2}$, $\mathbf{l}_{3}$, $\mathbf{l}_{4}$, respectively.
Set $\widetilde{\mathbf{h}}_1=f^*(\mathbf{h}_1)$ and $\widetilde{\mathbf{h}}_2=f^*(\mathbf{h}_2)$.
Let $\mathbf{c}_0$, $\mathbf{c}_1$, $\mathbf{c}_2$
be the curves in $|\widetilde{\mathbf{h}}_1+\widetilde{\mathbf{h}}_2-\widetilde{\mathbf{e}}_1-\widetilde{\mathbf{e}}_2-E|$,
$|\widetilde{\mathbf{h}}_1-E|$, $|\widetilde{\mathbf{h}}_2-E|$,~respectively.
Then
\begin{center}
$\widetilde{\mathbf{e}}_1$, $\widetilde{\mathbf{e}}_2$, $\widetilde{\mathbf{l}}_{1}$, $\widetilde{\mathbf{l}}_{2}$, $\widetilde{\mathbf{l}}_{3}$,
$\widetilde{\mathbf{l}}_{4}$, $\mathbf{c}_0$, $\mathbf{c}_1$, $\mathbf{c}_2$, $E$
\end{center}
are all $(-1)$-curves in $\widetilde{S}$. We compute $\widetilde{\tau}=2a-1$. Similarly, we see that
$$
\widetilde{P}(v)\sim_{\mathbb{R}} \left\{\aligned
&a\big(\widetilde{\mathbf{h}}_1+\widetilde{\mathbf{h}}_2\big)-\widetilde{\mathbf{e}}_1-\widetilde{\mathbf{e}}_2-vE \ \text{ if } 0\leqslant v\leqslant 2a-2, \\
&(3a-v-2)\big(\widetilde{\mathbf{h}}_1+\widetilde{\mathbf{h}}_2\big)+(1-2a+v)\big(\widetilde{\mathbf{e}}_1+\widetilde{\mathbf{e}}_2\big)+(2-2a)E\ \text{ if } 2a-2\leqslant v\leqslant a, \\
&(2a-1-v)\big(2\widetilde{\mathbf{h}}_1+2\widetilde{\mathbf{h}}_2-\widetilde{\mathbf{e}}_1-\widetilde{\mathbf{e}}_2-2E\big) \ \text{ if } a\leqslant v\leqslant 2a-1,
\endaligned
\right.
$$
and
$$
\widetilde{N}(v)=\left\{\aligned
&0\ \text{ if } 0\leqslant v\leqslant 2a-2, \\
&(v+2-2a)\mathbf{c}_0\ \text{ if } 2a-2\leqslant v\leqslant a, \\
&(v+2-2a)\mathbf{c}_0+(v-a)\big(\mathbf{c}_1+\mathbf{c}_2\big) \ \text{ if } a\leqslant v\leqslant 2a-1.
\endaligned
\right.
$$
This gives
$$
\widetilde{P}(v)^2=\left\{\aligned
&2a^2-v^2-2\ \text{ if } 0\leqslant v\leqslant 2a-2, \\
&2(a-1)(3a-2v-1)\ \text{ if } 2a-2\leqslant v\leqslant a, \\
&2(2a-1-v)^2\ \text{ if } a\leqslant v\leqslant 2a-1,
\endaligned
\right.
$$
and
$$
\widetilde{P}(v)\cdot E=\left\{\aligned
&v\ \text{ if } 0\leqslant v\leqslant 2a-2, \\
&2a-2\ \text{ if } 2a-2\leqslant v\leqslant a, \\
&4a-2v-2\ \text{ if } a\leqslant v\leqslant 2a-1.
\endaligned
\right.
$$
Now, integrating, we get $S_D(E)=\frac{a^2+a-1}{a+1}$.

Let $O$ be a point in $E$. Then
$$
S\big(W^E_{\bullet,\bullet};O\big)=
\frac{1}{a^2-1}\int\limits_{2a-2}^{2a-1}\mathrm{ord}_O\big(\widetilde{N}(v)\vert_{E}\big)\big(\widetilde{P}(v)\cdot E\big)dv+\frac{2(a-1)}{a+1}.
$$
Thus, if $O\not\in\mathbf{c}_0\cup\mathbf{c}_1\cup\mathbf{c}_2$, then $S(W^E_{\bullet,\bullet};O)=\frac{2(a-1)}{a+1}$.
Similarly, if $O\in\mathbf{c}_0$, then
$$
S\big(W^E_{\bullet,\bullet};O\big)=
\frac{1}{a^2-1}\int\limits_{2a-2}^{2a-1}(v+4-2a)\big(\widetilde{P}(v)\cdot E\big)dv+\frac{2(a-1)}{a+1}=\frac{a^2+a+1}{3(a+1)}.
$$
Likewise, if $O\in\mathbf{c}_1\cup\mathbf{c}_2$, then
$$
S\big(W^E_{\bullet,\bullet};O\big)=
\frac{1}{a^2-1}\int\limits_{a}^{2a-1}(v + 1 - a)\big(\widetilde{P}(v)\cdot E\big)dv+\frac{2(a-1)}{a+1}=\frac{(a+5)(a-1)}{3(a+1)}.
$$
Therefore, using \eqref{equation:AZ-surface-blow-up}, we get
$$
\delta_P(S,D)\geqslant\min\Bigg\{\frac{2(a+1)}{a^2+a-1},\frac{3(a+1)}{a^2+a+1},\frac{3(a+1)}{(a+5)(a-1)}\Bigg\},
$$
which implies the required assertion.
\end{proof}

\end{document}